\documentclass[11pt,reqno]{amsart}

\usepackage{amsfonts,latexsym,amsthm,amssymb,amsmath,amscd,euscript}
\usepackage{framed}
\usepackage{fullpage}
\usepackage{hyperref}
    \hypersetup{colorlinks=true,citecolor=blue,urlcolor =black,linkcolor=gray,linkbordercolor={1 0 0}}
\usepackage{tikz-cd}
\usepackage{comment}
\usepackage{mathtools}
\usepackage{mathdots}
\usepackage{enumerate}
\usepackage[new]{old-arrows}
\usepackage{mathrsfs}
\usepackage{pgfplots}
\usepackage{cancel}
\usepackage{rotating}
\usepackage{bookmark}
\usepackage{schemata}
\usepackage{environ}
\usepackage{graphics}
\usepackage{graphicx}
\usepackage{xfrac}
\usepackage{nicefrac}
\usepackage{tensor}
\usepackage{expdlist}
\usepackage{stmaryrd}
\usepackage{CJKutf8}

\usetikzlibrary{positioning}
\DeclareFontFamily{U}{mathx}{\hyphenchar\font45}
\DeclareFontShape{U}{mathx}{m}{n}{
      <5> <6> <7> <8> <9> <10>
      <10.95> <12> <14.4> <17.28> <20.74> <24.88>
      mathx10
      }{}
\DeclareSymbolFont{mathx}{U}{mathx}{m}{n}
\DeclareFontSubstitution{U}{mathx}{m}{n}
\DeclareMathAccent{\widecheck}{0}{mathx}{"71}
\DeclareMathAccent{\wideparen}{0}{mathx}{"75}

\makeatletter
\DeclareFontFamily{U}{tipa}{}
\DeclareFontShape{U}{tipa}{m}{n}{<->tipa10}{}
\newcommand{\arc@char}{{\usefont{U}{tipa}{m}{n}\symbol{62}}}%

\newcommand{\arc}[1]{\mathpalette\arc@arc{#1}}

\newcommand{\arc@arc}[2]{%
  \sbox0{$\m@th#1#2$}%
  \vbox{
    \hbox{\resizebox{\wd0}{\height}{\arc@char}}
    \nointerlineskip
    \box0
  }%
}
\makeatother

\makeatletter
\def\widebreve{\mathpalette\wide@breve}
\def\wide@breve#1#2{\sbox\z@{$#1#2$}%
     \mathop{\vbox{\m@th\ialign{##\crcr
\kern0.08em\brevefill#1{0.8\wd\z@}\crcr\noalign{\nointerlineskip}%
                    $\hss#1#2\hss$\crcr}}}\limits}
\def\brevefill#1#2{$\m@th\sbox\tw@{$#1($}%
  \hss\resizebox{#2}{\wd\tw@}{\rotatebox[origin=c]{90}{\upshape(}}\hss$}
\makeatletter

\usepackage{color}
\definecolor{blaze}{rgb}{0.00,0.420,0.00}

\DeclareRobustCommand\longtwoheadrightarrow{\relbar\joinrel\twoheadrightarrow}
\allowdisplaybreaks[1]

\NewEnviron{Thm*}[1][]{
	\vspace{10pt}\newline\noindent\ignorespaces\hspace*{-12.5pt}\DoBrackets\schema{}{\vspace{-7.5pt}\begin{theorem*}[#1]\BODY\newline\noindent$\square$\end{theorem*}\vspace{-7.5pt} }\ignorespacesafterend\leavevmode\vspace{12.5pt}\newline
}
\NewEnviron{Thm}[1][]{
	\vspace{10pt}\newline\noindent\ignorespaces\hspace*{-12.5pt}\DoBrackets\schema{}{\vspace{-7.5pt}\begin{theorem}[#1]\BODY\newline\noindent$\square$\end{theorem}\vspace{-7.5pt} }\ignorespacesafterend\leavevmode\vspace{12.5pt}\newline
}
\NewEnviron{Prop*}[1][]{
	\vspace{10pt}\newline\noindent\ignorespaces\hspace*{-12.5pt}\DoBrackets\schema{}{\vspace{-7.5pt}\begin{proposition*}[#1]\BODY\newline\noindent$\square$\end{proposition*}\vspace{-7.5pt} }\ignorespacesafterend\leavevmode\vspace{12.5pt}\newline
}
\NewEnviron{Prop}[1][]{
	\vspace{10pt}\newline\noindent\ignorespaces\hspace*{-12.5pt}\DoBrackets\schema{}{\vspace{-7.5pt}\begin{proposition}[#1]\BODY\newline\noindent$\square$\end{proposition}\vspace{-7.5pt} }\ignorespacesafterend\leavevmode\vspace{12.5pt}\newline
}
\NewEnviron{Def*}[1][]{
	\vspace{10pt}\newline\noindent\ignorespaces\hspace*{-7.5pt}\DoParens\schema{}{\vspace{-7.5pt}\begin{definition*}[#1]\BODY\end{definition*}\vspace{-7.5pt} }\ignorespacesafterend\leavevmode\vspace{12.5pt}\newline
}
\NewEnviron{Def}[1][]{
	\vspace{5pt}\newline\noindent\ignorespaces\hspace*{-7.5pt}\DoParens\addtocounter{definition}{-1}\schema{}{\vspace{-7.5pt}\begin{definition}[#1]\BODY\end{definition}\vspace{-7.5pt} }\ignorespacesafterend\leavevmode\vspace{7.5pt}\newline
}
\NewEnviron{Lem*}[1][]{
	\vspace{10pt}\newline\noindent\ignorespaces\hspace*{-12.5pt}\DoBrackets\schema{}{\vspace{-7.5pt}\begin{lemma*}[#1]\BODY\newline\noindent$\square$\end{lemma*}\vspace{-7.5pt} }\ignorespacesafterend\leavevmode\vspace{12.5pt}\newline
}
\NewEnviron{Lem}[1][]{
	\vspace{10pt}\newline\noindent\ignorespaces\hspace*{-12.5pt}\DoBrackets\schema{}{\vspace{-7.5pt}\begin{lemma}[#1]\BODY\newline\noindent$\square$\end{lemma}\vspace{-7.5pt} }\ignorespacesafterend\leavevmode\vspace{12.5pt}\newline
}
\NewEnviron{Fact}[1][]{
	\vspace{10pt}\newline\noindent\ignorespaces\hspace*{-12.5pt}\DoBrackets\schema{}{\vspace{-7.5pt}\begin{fact}[#1]\BODY\newline\noindent$\square$\end{fact}\vspace{-7.5pt} }\ignorespacesafterend\leavevmode\vspace{12.5pt}\newline
}
\NewEnviron{Fact*}[1][]{
	\vspace{10pt}\newline\noindent\ignorespaces\hspace*{-12.5pt}\DoBrackets\schema{}{\vspace{-7.5pt}\begin{fact*}[#1]\BODY\newline\noindent$\square$\end{fact*}\vspace{-7.5pt} }\ignorespacesafterend\leavevmode\vspace{12.5pt}\newline
}

\NewEnviron{THM*}[1][]{
	\vspace{10pt}\newline\noindent\ignorespaces\hspace*{-12.5pt}\DoBrackets\schema{}{\vspace{-7.5pt}\begin{theorem*}[#1]\BODY\end{theorem*}\vspace{-6pt} }\ignorespacesafterend\leavevmode\vspace{12.5pt}\newline
}
\NewEnviron{FACT*}[1][]{
	\vspace{10pt}\newline\noindent\ignorespaces\hspace*{-12.5pt}\DoBrackets\schema{}{\vspace{-7.5pt}\begin{fact*}[#1]\BODY\end{fact*}\vspace{-6pt} }\ignorespacesafterend\leavevmode\vspace{12.5pt}\newline
}
\NewEnviron{LEM*}[1][]{
	\vspace{10pt}\newline\noindent\ignorespaces\hspace*{-12.5pt}\DoBrackets\schema{}{\vspace{-7.5pt}\begin{lemma*}[#1]\BODY\end{lemma*}\vspace{-6pt} }\ignorespacesafterend\leavevmode\vspace{12.5pt}\newline
}
\NewEnviron{PROP*}[1][]{
	\vspace{10pt}\newline\noindent\ignorespaces\hspace*{-12.5pt}\DoBrackets\schema{}{\vspace{-7.5pt}\begin{proposition*}[#1]\BODY\end{proposition*}\vspace{-6pt} }\ignorespacesafterend\leavevmode\vspace{12.5pt}\newline
}
\NewEnviron{THM}[1][]{
	\vspace{10pt}\newline\noindent\ignorespaces\hspace*{-12.5pt}\DoBrackets\addtocounter{theorem}{-1}\schema{}{\vspace{-7.5pt}\begin{theorem}[#1]\BODY\end{theorem}\vspace{-6pt} }\ignorespacesafterend\leavevmode\vspace{12.5pt}\newline
}
\NewEnviron{FACT}[1][]{
	\vspace{10pt}\newline\noindent\ignorespaces\hspace*{-12.5pt}\DoBrackets\addtocounter{theorem}{-1}\schema{}{\vspace{-7.5pt}\begin{fact}[#1]\BODY\end{fact}\vspace{-6pt} }\ignorespacesafterend\leavevmode\vspace{12.5pt}\newline
}
\NewEnviron{LEM}[1][]{
	\vspace{10pt}\newline\noindent\ignorespaces\hspace*{-12.5pt}\DoBrackets\addtocounter{theorem}{-1}\schema{}{\vspace{-7.5pt}\begin{lemma}[#1]\BODY\end{lemma}\vspace{-6pt} }\ignorespacesafterend\leavevmode\vspace{12.5pt}\newline
}
\NewEnviron{PROP}[1][]{
	\vspace{10pt}\newline\noindent\ignorespaces\hspace*{-12.5pt}\DoBrackets\addtocounter{theorem}{-1}\schema{}{\vspace{-7.5pt}\begin{proposition}[#1]\BODY\end{proposition}\vspace{-6pt} }\ignorespacesafterend\leavevmode\vspace{12.5pt}\newline
}

\NewEnviron{Obs*}[1][]{
	\vspace{10pt}\newline\noindent\ignorespaces\hspace*{-7.5pt}\DoParens\schema{}{\vspace{-7.5pt}\begin{observation*}[#1]\BODY\end{observation*}\vspace{-7.5pt} }\ignorespacesafterend\leavevmode\vspace{12.5pt}\newline
}

\theoremstyle{theorem}
\newtheorem*{theorem*}{Theorem}%
\newtheorem*{proposition*}{Proposition}
\newtheorem{theorem}{Theorem}
\newtheorem{proposition}[theorem]{Proposition}
\newtheorem{lemma}[theorem]{Lemma}
\newtheorem*{lemma*}{Lemma}
\newtheorem{fact}[theorem]{Fact}
\newtheorem*{fact*}{Fact}

\newtheorem{definition}{Definition}
\newtheorem*{definition*}{Definition}

\newtheorem*{example*}{Example}
\newtheorem*{claim*}{Claim}

\newtheorem*{result*}{Result}
\newtheorem*{observation*}{Observation}
\newtheorem*{question*}{Question}
\newtheorem*{direction*}{Direction}
\newtheorem*{conjecture*}{Conjecture}

\theoremstyle{remark}

\newcommand{\BC}{\mathbb C}

\newcommand{\BZ}{\mathbb Z}
\newcommand{\BN}{\mathbb N}

\DeclareMathOperator{\img}{Img}
\DeclarePairedDelimiter\wangle{\langle}{\rangle}

\newcommand{\nc}{\newcommand}

\nc{\on}{\operatorname}
\nc{\spec}{\on{Spec}}

\DeclarePairedDelimiter\pr{(}{)}

\DeclareMathOperator{\Spec}{Spec}
\renewcommand{\spec}{\Spec}

\newcommand{\ol}[1]{\overline{#1}}

\nc{\wt}[1]{\widetilde{#1}}
\nc{\wh}[1]{\widehat{#1}}
\nc{\pd}{\partial}
\nc{\cal}[1]{\mathcal{#1}}
\renewcommand{\frak}[1]{\mathfrak{#1}}
\nc{\scr}[1]{\mathscr{#1}}
\nc{\tsubseteq}[1]{\overset{\textnormal{#1}}{\subseteq}}
\nc{\tsupseteq}[1]{\overset{\textnormal{#1}}{\supseteq}}
\nc{\tsubset}[1]{\overset{\textnormal{#1}}{\subset}}
\nc{\tsupset}[1]{\overset{\textnormal{#1}}{\supset}}
\nc{\vphi}{\varphi}
\nc{\beau}{\displaystyle}
\nc{\existss}{\exists\;}
\nc{\foralls}{\forall\;}

\DeclarePairedDelimiter\set{\{}{\}}
\nc{\ov}[1]{\overrightarrow{#1}}
\nc{\cdotsc}{\cdots\hspace{-0.1pt}}

\nc{\D}{\textnormal{D}}

\nc{\tn}[1]{\textnormal{#1}}
\nc{\lto}{\longrightarrow}
\nc{\lmto}{\longmapsto}

\nc{\sm}{\setminus}

\nc{\scomp}{\textsf{c}}
\DeclareMathOperator{\cnt}{cnt}

\nc{\ncnt}{\lnot\cnt}
\nc{\ucong}{\overset{!}{\cong}}
\nc{\tcong}[1]{\overset{#1}{\cong}}
\nc{\coker}{\on{Coker}}
\newcommand{\wc}[1]{\widecheck{#1}}
\DeclarePairedDelimiter\floor{\lfloor}{\rfloor}

\nc{\eps}{\varepsilon}
\nc{\BS}{\mathbb{S}}
\nc{\BT}{\mathbb{T}}
\nc{\SL}{\mathscr{L}}
\nc{\SC}{\mathscr{C}}

\DeclareMathOperator{\tor}{Tor}
\nc{\wa}[1]{\wideparen{#1}}
\nc{\qidl}{\frak{q}}
\nc{\nidl}{\frak n}
\DeclareMathOperator{\Frac}{Frac}
\nc{\te}[1]{\textnormal{#1}}
\nc{\imps}{\quad\quad\ \ \,}
\nc{\comp}{\complement}
\nc{\van}{\cal V}
\nc{\ide}{\cal I}
\nc{\nfrac}[2]{\nicefrac{#1}{#2}}
\nc{\snfrac}[2]{\scalebox{1.35}{$\nfrac{#1}{#2}$}}
\nc{\mnfrac}[2]{\scalebox{1.4}[1.5]{$\nfrac{#1}{#2}$}}%1.35
\nc{\inj}{\hookrightarrow}
\nc{\surj}{\twoheadrightarrow}
\nc{\linj}{\longhookrightarrow}
\nc{\lsurj}{\longtwoheadrightarrow}
\nc{\four}{\cal F}
\nc{\urot}[1]{\mathbin{\rotatebox[origin=c]{90}{$#1$}}}
\nc{\drot}[1]{\mathbin{\rotatebox[origin=c]{-90}{$#1$}}}
\nc{\vrot}[2]{\mathbin{\rotatebox[origin=c]{#2}{$#1$}}}

\nc{\wb}[1]{\widebreve{#1}}

\nc{\BH}{\mathbb{H}}

\nc{\T}{\textnormal{T}}
\nc{\N}{\vrot{\T}{180}\!}
\nc{\blt}{\bullet}
\nc{\lbd}{\lambda}

\renewcommand{\tt}[1]{\texttt{#1}}
\renewcommand{\sf}[1]{\mathsf{#1}}
\nc{\sq}{{\mathop{\square}}}
\nc{\onl}[1]{\operatorname*{#1}}
\nc{\bigger}[1]{\scalebox{1.5}{$#1$}}
\nc{\tbt}[4]{\begin{pmatrix}#1&#2\\#3&#4\end{pmatrix}}
\nc{\thbth}[9]{\begin{pmatrix}#1&#2&#3\\#4&#5&#6\\#7&#8&#9\end{pmatrix}}
\nc{\tbo}[2]{\begin{pmatrix}#1\\#2\end{pmatrix}}
\nc{\thbo}[3]{\begin{pmatrix}#1\\#2\\#3\end{pmatrix}}
\nc{\tsq}[3]{\tensor{{#1}}{^{#2}_{#3}}}
\nc{\rnk}{\on{rnk}}
\nc{\Ad}{\on{Ad}}
\nc{\ad}{\on{ad}}
\nc{\glie}{\frak{g}}
\nc{\SO}{\on{SO}}
\nc{\sk}{\on{Sk}}
\nc{\snakeanchor}{\ar[draw=none]{d}[name=X, anchor=center]{}}
\nc{\snakearrow}[1]{\ar[rounded corners,
	to path={ -- ([xshift=2ex]\tikztostart.east)
		|- (X.center) \tikztonodes
		-| ([xshift=-2ex]\tikztotarget.west)
		-- (\tikztotarget)}]{dll}{#1}}
\nc{\cosnakearrow}[1]{\ar[rounded corners,
	to path={ -- ([xshift=-2ex]\tikztostart.west)
		|- (X.center) \tikztonodes
		-| ([xshift=2ex]\tikztotarget.east)
		-- (\tikztotarget)}]{urr}{#1}}
\nc{\colim}{\on*{colim}}
\nc{\sing}{\on{Sing}}
\makeatletter
\newcommand{\dircolim@}[2]{%
  \vtop{\m@th\ialign{##\cr
    \hfil$#1\operator@font colim$\hfil\cr
    \noalign{\nointerlineskip\kern1.5\ex@}#2\cr
    \noalign{\nointerlineskip\kern-\ex@}\cr}}%
}
\newcommand{\dircolim}{%
  \mathop{\mathpalette\dircolim@{\rightarrowfill@\textstyle}}\nmlimits@
}
\makeatother
\makeatletter
\newcommand{\xdashrightarrow}[2][]{\ext@arrow 0359\rightarrowfill@@{#1}{#2}}
\newcommand{\xdashleftarrow}[2][]{\ext@arrow 3095\leftarrowfill@@{#1}{#2}}
\newcommand{\xdashleftrightarrow}[2][]{\ext@arrow 3359\leftrightarrowfill@@{#1}{#2}}
\def\rightarrowfill@@{\arrowfill@@\relax\relbar\rightarrow}
\def\leftarrowfill@@{\arrowfill@@\leftarrow\relbar\relax}
\def\leftrightarrowfill@@{\arrowfill@@\leftarrow\relbar\rightarrow}
\def\arrowfill@@#1#2#3#4{%
  $\m@th\thickmuskip0mu\medmuskip\thickmuskip\thinmuskip\thickmuskip
   \relax#4#1
   \xleaders\hbox{$#4#2$}\hfill
   #3$%
}
\makeatother
\nc{\lcm}{\on{lcm}}
\nc{\Stab}{\on{Stab}}
\nc{\Nbhd}{\on{Nbhd}}
\nc{\Lie}{\on{Lie}}
\nc{\gl}{\on{\mathfrak{gl}}}
\nc{\Gr}{\on{Gr}}
\nc{\act}{\on{act}}
\nc{\ev}{\on{ev}}
\nc{\Sq}{\on{Sq}}
\nc{\ext}{\on{Ext}}
\nc{\Ext}{\on{Ext}}
\nc{\Dist}{\on{Dist}}
\nc{\dlie}{\frak{d}}
\nc{\hlie}{\frak{h}}
\nc{\pt}{\te{pt}}
\nc{\qedcheck}{\vspace{-1em}\begin{flushright}$\checkmark\quad$\end{flushright}\vspace{-0.5em}}
\nc{\lemproof}[1]{\smallskip\noindent\underline{#1:}}
\nc{\F}{\te{F}}
\nc{\jota}{\jmath}
\nc{\Tor}{\tor}
\nc{\Fl}{\on{Fl}}
\nc{\prob}{\on{prob}}
\nc{\didl}{\frak{d}}

\nc{\Rep}{\on{\sf{Rep}}}
\nc{\Irrep}{\on{\sf{irRep}}}
\nc{\Res}{\on{Res}}
\nc{\Ind}{\on{Ind}}
\nc{\aut}{\on{aut}}
\nc{\Der}{\on{Der}}

\nc{\so}{\on{\frak{so}}}

\nc{\Img}{\img}
\nc{\Sidl}{\frak{S}}
\nc{\Slie}{\frak{S}}
\nc{\iita}{\imath}
\nc{\hidl}{\hlie}
\nc{\blie}{\frak{b}}
\nc{\bidl}{\frak{b}}
\nc{\nlie}{\nidl}
\nc{\TQ}{\text{Q}}
\nc{\tq}{\text{q}}
\nc{\Q}{\text{Q}}
\nc{\larr}{\leftarrow}
\nc{\darr}{\downarrow}
\nc{\nwarr}{\nwarrow} 
\nc{\searr}{\searrow} 
\nc{\Tot}{\on{Tot}}
\nc{\twt}{\on{wt}}
\nc{\weit}{\on{wt}}
\nc{\sfcong}[1]{\overset{\sf{#1}}{\cong}}
\nc{\vtheta}{\vartheta}
\nc{\JH}{\on{JH}}
\nc{\typ}{\on{typ}}
\nc{\alie}{\frak{a}}
\nc{\rem}{\on{rem}}
\nc{\rmod}{\on{mod}}
\nc{\Ray}{\on{Ray}}
\nc{\Col}{\on{Col}}
\DeclarePairedDelimiter\bigpr{\big(}{\big)}

\DeclarePairedDelimiter\Bigpr{\Big(}{\Big)}

\nc{\llra}{\longleftrightarrow}
\nc{\Llra}{\Longleftrightarrow} 
\nc{\lla}{\longleftarrow} 
\nc{\lra}{\longrightarrow} 
\nc{\Lra}{\Longrightarrow}  
\newcommand\sqplus{\mathbin{\ooalign{$\sqcup$\cr%
   \hfil\raise0.42ex\hbox{$\scriptscriptstyle+$}\hfil\cr}}}
\newcommand\sqminus{\mathbin{\ooalign{$\sqcup$\cr%
   \hfil\raise0.42ex\hbox{$\scriptscriptstyle-$}\hfil\cr}}}
\newcommand\uminus{\mathbin{\ooalign{$\cup$\cr%
   \hfil\raise0.42ex\hbox{$\scriptscriptstyle-$}\hfil\cr}}}
\nc{\bsfrac}[2]{\scalebox{1.4}[1.5]{$\sfrac{#1}{#2}$}}
\newcommand\quot[2]{
        \mathchoice
            {% \displaystyle
                \text{\raise1ex\hbox{$#1$}\Big/\lower1ex\hbox{$#2$}}%
            }
            {% \textstyle
                #1\,/\,#2
            }
            {% \scriptstyle
                #1\,/\,#2
            }
            {% \scriptscriptstyle  
                #1\,/\,#2
            }
}

\title{On Eventually Periodic Sets as Minimal Additive Complements} % IMPORTANT: Change the problemset number as needed.
\pgfplotsset{compat=1.14}
\begin{document}

\maketitle

\vspace*{-0.25in}
\centerline{Fan Zhou}
\centerline{Harvard College}
% Just so that your CA's can come knocking on your door when you don't hand in that problemset on time...
%\centerline{Thayer 504}
\centerline{\href{mailto:fanzhou@college.harvard.edu}{{\tt{ fanzhou@college.harvard.edu}}}}
\vspace*{0.15in}

%\begin{framed}
%  Note: This is your \LaTeX~template. Feel free to use it on your problemsets (recommended as a means of helping your CA's grade). For Windows and Macintosh users, try {\TeX}Studio as a \LaTeX~editor. If this program ever doesn't work, you can use online {\TeX}ers like write\LaTeX~, verb\TeX, and \href{http://www.overleaf.com}{Overleaf}. Any questions? Contact us at \href{mailto:vikramsundar@college.harvard.edu}{{\tt vikramsundar@college.harvard.edu}} and \href{mailto:prasad01@college.harvard.edu}{{\tt prasad01@college.harvard.edu}}. Please \TeX your problem sets.

%    Have fun \TeX-ing! Oh, and delete this box when you've understood this!
%\end{framed}
%\textit{Inverse problems for minimal complements and maximal supplements}

\footnotesize{We say a subset $C$ of an abelian group $G$ \textit{arises as a minimal additive complement} if there is some other subset $W$ of $G$ such that $C+W=\{c+w:c\in C,\ w\in W\}=G$ and such that there is no proper subset $C'\supset C$ such that $C'+W=G$. In their recent paper, Burcroff and Luntzlara studied, among many other things, the conditions under which ``eventually periodic sets'', which are finite unions of infinite (in the positive direction) arithmetic progressions and singletons, arise as minimal additive complements in $\BZ$. In the present paper we shall study this question further. We give, in the form of bounds on the period $m$, some sufficient conditions for an eventually periodic set to be a minimal additive complement; in particular we show that ``all eventually periodic sets are eventually minimal additive complements''. Moreover, we generalize this to a framework in which ``patterns'' of points are projected down to $\BZ$, and we show that all sets which arise this way are eventually minimal additive complements. We also introduce a formalism of formal power series, which serves purely as a bookkeeper in writing down proofs. Through our work we are able to answer a question of Burcroff and Luntzlara in a large class of cases.}

\normalsize
% \tableofcontents
% \vspace{-2.5em}
\section{Introduction}
    The setting of the question is as follows. For $C$ and $W$ subsets of an abelian group $G$, we say that $C$ is an \textit{additive complement} to $W$ if the Minkowski sum $C+W$ is equal to $G$, i.e. if
    \[G=C+W\coloneqq \{c+w:c\in C,\ w\in W\}.\]
    $C$ is a \textit{minimal additive complement} (or \textit{MAC}) to $W$ if there is no proper subset of $C$ which is an additive complement to $W$. We say $C$ \textit{arises as a MAC} (or \textit{is a MAC}) if there exists a $W$ to which $C$ is a MAC.
 
    In particular, we will be interested in sets which are called ``eventually periodic sets'', which are defined as follows: For $S\subseteq\BZ$, let $S_{/m}$ denote the image in $\BZ/m\BZ$ under the standard projection. An \textit{eventually periodic set} of period $m$ is a set of integers of form\footnote{Here we adopt the standard where $\BN$ includes 0, i.e. $\BN=\BZ_{\ge 0}$.}
    \[(m\BN+A)\cup B\cup F\]
    where $A$, $B$, and $F$ are finite, $A$ is nonempty, $B_{/m}\subseteq A_{/m}$, and $F_{/m}\cap A_{/m}=\emptyset$.

    In this paper, we shall study the conditions under which eventually periodic sets arise as MACs. We will show two main theorems, namely Theorems 4 and 7, the latter of which is stated in a new general framework of ``patterns'' which we introduce, and use them to deduce several other results, for example that 
    \begin{result*}[Proposition 1]
        Any eventually periodic set $C=m\BN\cup B\cup\{f\}$ (where $B$ is a finite subset of $\BZ$ such that $b\equiv 0\mod m$ for all $b\in B$ and $f\not\equiv 0\mod m$) arises as a MAC in $\BZ$.
    \end{result*}
    \noindent and that 
    \begin{result*}[Proposition 2]
        Any $C=m\BN\cup F$ (for $F$ in a single congruence class mod $m$) arises as a MAC in $\BZ$.
    \end{result*}
    The two main theorems are of roughly the following form: if there exists some set cover of $m\BN\cup B$ (or more generally a congruence class mod $m$ in which $C$ has infinitely many points) satisfying certain conditions, and if $m$ is greater than a certain bound, then $C$ arises as a MAC. In particular, roughly speaking, ``all eventually periodic sets are eventually MACs''. We also introduce a formalism of ``formal power series'' to help reduce the process of checking the proofs of such statements to routine calculations. We should however issue a disclaimer that there is no new content in these formal power series, and that no tools or clever tricks from the broader theory of generating functions/formal power series will be utilized; these formal power series here serve purely as bookkeepers.
    
    % In this paper, we shall study the conditions under which eventually periodic sets arise as MACs. We will show two main theorems, namely 4 and 7, the latter of which is stated in a new general framework of ``patterns'' which we introduce, and use them to deduce several other results, for example that any $m\BN\cup B\cup\{f\}$ arises as a MAC, and that any $m\BN\cup F$ (for $F$ in a single congruence class mod $m$) arises as a MAC. The two main theorems are of roughly the following form: if there exists some set cover of $m\BN\cup B$ (or more generally a congruence class in mod $m$ in which $C$ has infinitely many points) satisfying certain conditions, and if $m$ is greater than a certain bound, then $C$ arises as a MAC. In particular, roughly speaking, ``all eventually periodic sets are eventually MACs''. We also introduce a formalism of ``formal power series'' to help reduce the process of checking the proofs of such statements to routine calculations. We should however issue a disclaimer that there is no new content in these formal power series, and that no tools or clever tricks from the broader theory of generating functions/formal power series will be utilized; these formal power series here serve purely as bookkeepers.
    
    \subsection{Background and Motivation}
    Minimal additive complements were introduced in 2011 as an arithmetic analogue to the metric concept of \textit{$h$-nets} in groups by Nathanson [N], who showed that every nonempty finite subset of $\BZ$ has a MAC, and moreover that every additive complement of such a set contains a MAC. The question of which subsets of $\BZ$ have MACs has since been studied by many; for example, Chen and Yang [CY] showed in 2012 that all subsets of $\BZ$ which are unbounded both above and below have MACs. Kiss, S\'andor, and Yang [KSY] in 2019 introduced the concept of ``eventually periodic sets'' and studied the question of when these sets have MACs. 
    
    The natural ``inverse problem'' is then to study which subsets of $\BZ$ arise as MACs. This study was initiated by Kwon [K] in 2019, who showed that every nonempty finite set in $\BZ$ arises as a MAC. Alon, Kravitz, and Larson [AKL] extended this further in 2020 and showed that, in any finite abelian group $G$, any nonempty subset $C$ with size bounded above by some constant depending on $|G|$ will arise as a MAC. Moreover, they showed that any nonempty finite subset of an infinite abelian group arises as a MAC. Later in 2020, Biswas and Saha [BSb] generalized this even further and showed that, for any group $G$ (abelian or not), any nonempty finite subset $C$ with $|G|>|C|^5-|C|^4$ will arise as a MAC; in particular, any nonempty finite subset of any infinite group will arise as a MAC. Also in 2020, Biswas and Saha [BSa] derived some conditions for subsets to \textit{not} arise as MACs in an arbitrary group; for example, their results show that $(\{3, 5, 7, 9, 11\} + 12\BZ) \cup \{p\te{ prime} : p \equiv 1 \mod 12\}$ is \textit{not} a MAC in $\BZ$.

    Our motivation is as follows. In 2020, Burcroff and Luntzlara [BL] studied (among many other things) the eventually periodic sets of Kiss, S\'andor, and Yang and the question of when do they arise as MACs. They showed that there are three certain necessary conditions for an eventually periodic set to arise as a MAC, whose statements are rather technical and which we therefore skip here. In the case that $m$ is prime, they gave a fourth necessary condition. They also showed that, in certain circumstances, namely for $m$ a prime congruent to 2 modulo 3 and for $A_{/m}$ and $F_{/m}$ of certain prescribed sizes ($\frac{m+1}{3}$ and $1$, respectively), these necessary conditions are also sufficient. Using these facts in conjunction, they showed that any set of form $2\BN\cup B\cup F$ arises as a MAC if and only if $2\BZ\sm(2\BN\cup B)=F+W$ for some $W\subseteq\BZ$. In conclusion, BL have given necessary and sufficient conditions for eventually periodic sets of prime period congruent to 2 mod 3 and with $|F_{/m}|=1$ and $|A_{/m}|=\frac{p+1}{3}$ to arise as MACs. They also simplified these conditions to something very concise and concrete in the case of $m=2$. A natural question to ask then is if there are other circumstances in which an eventually periodic set $C$ arises as a MAC. We will attempt to treat this direction in this paper.
    
    By venturing in this direction, we are able to answer an open question of Burcroff and Luntzlara for a large number of cases. At the end of their paper [BL], they raised the following question:
    \begin{question*}
        Which sets of the form $C_1\cup (-C_2)$, where $C_1$ and $C_2$ are eventually periodic sets of integers, arise as minimal additive complements?
    \end{question*}
\noindent    Our results (in particular, Theorem 7) will show that a large class of such sets do indeed rise as minimal additive complements.
    
    Let us briefly explain why we study the objects we do in this paper. The existence of a nonempty $F$ is crucial, as it allows us to set up so-called ``dependent elements'' in our constructions later (to be explained later; roughly this is just to ensure that all elements of the eventually periodic set are necessary). Without this $F$, it is for example easy to see that $\BN$ is not a MAC in $\BZ$. Similarly, our results here are true roughly because if $m$ is sufficiently large, then there is ``enough space to maneuver'' in setting up the ``dependent elements''.

\section{Prelude}
\subsection{Preliminary Definitions}%For the sake of brevity and clarity, liberal use will be made of logical symbols such as $\forall,\ \exists,\ :$, where $:$ is short for ``such that''.
    First some brief notes on convention: we will take the ``natural numbers'' $\BN=\BZ_{0+}=\{n\in\BZ:n\ge 0\}$ to include zero. Negative integers are denoted $\BZ_-$. We will write $\on{mod}_nk$ to denote the remainder of $k$ when divided by $n$. Some further notes on hats: throughout this paper we will decorate many of our symbols with hats, namely the widehat $\wh{\sq }$, the overline $\ol{\sq }$, and the widetilde $\wt{\sq }$; our philosophy is that the widehat denotes lifts, the overline is any generic marker (preferably to do with some ``natural'' modification), and the widetilde denotes quotients. 

    Recall from the introduction that the general setting of the question is as follows: 
    \begin{definition}
    For $C$ and $W$ subsets of an abelian group $G$, we say $C$ is an \textit{additive complement} to $W$ if
    \[G=C+W\coloneqq \{c+w:c\in C,w\in W\}.\]
    
    $C$ is a \textit{minimal additive complement} (shortened to \textit{MAC}) to $W$ if there is no proper subset of $C$ which is an additive complement to $W$.
    
    We say $C$ \textit{arises as a MAC} (or \textit{is a MAC}) if there exists a $W$ to which $C$ is a MAC.
    
    For $W=\{w\}$ a singleton, we will denote
    \[C+w\coloneqq C+\{w\}.\]
    
    If $A\subseteq S\subseteq G$ are subsets such that there exists $B\subseteq G$ with
    \[S=A+B,\]
    then we say $S$ is \textit{coverable by} $A$.
    \end{definition}
    % Perhaps we should add that, sometimes, rather than write the cumbersome $C+\{w\}$ when $W$ is a singleton, we may write simply $C+w$ instead; we might call this ``the translation of $C$ by $w$'', or simply ``a translate of $C$''.
    
    For this paper we will be concerned mostly with the case of $G=\BZ$. There, many results are known already; for example in this paper we will utilize the theorem of Kwon [K] mentioned in the introduction:
    \begin{theorem*}[K]
    All finite subsets of $\BZ$ arise as MACs.
    \end{theorem*}
    Closely related to this is a lemma by Burcroff-Luntzlara [BL] which we shall also employ later,  whereupon we shall call it the ``BL lemma''.
    \begin{lemma*}[BL]
    For a fixed finite set $F\subset\BZ$ and $W\subseteq\BZ$ such that $F+W\supset\BN$, there exists a set $W'\subseteq\BZ$ such that $F+W'=F+W$ and $F'+W'\neq F+W'$ for any proper subset $F'\subset F$.
    \end{lemma*}
    % We shall make use of this lemma later, whereupon we shall call it the ``BL lemma''.
    
    Recall that the results of this paper are concerned with ``eventually periodic sets'', defined by
    \begin{definition}
        For $S\subseteq\BZ$, let $S_{/m}$ denote the image in $\BZ/m\BZ$ under projection.
        
        An \textit{eventually periodic set} of period $m$ is a set of integers of form
    \[(m\BN+A)\cup B\cup F\]
    where $A$ is nonempty, $B$ and $F$ are finite, $B_{/m}\subseteq A_{/m}$, and $F_{/m}\cap A_{/m}=\emptyset$.
    
    As noted by B-L, WLOG we may take $A$ to have at most one element in each congruence class mod $m$.
    \end{definition}
    In other words, these are sets which have infinite (in the positive direction) arithmetic progressions of period $m$ starting from all elements of $A$, and various finite ``exceptions'' $B$ and $F$, where elements of $B$ lie in the same congruence classes as elements of $A$ and elements of $F$ lie in different congruence classes than those of $A$. %Obviously, and as Burcroff-Luntzlara have noted, for such sets it suffices to consider the case where $A$ has at most a single element in each congruence class.
    
    % In their paper, Burcroff-Luntzlara considered (among many other things) the question of when these eventually periodic sets arose as MACs. They showed that there are three certain necessary conditions for a set to arise as a MAC, whose statements are rather technical and which we therefore skip here. In the case that $m$ is prime, they gave a fourth necessary condition. They also showed that, in certain circumstances, namely for $m$ a prime congruent to 2 modulo 3 and for $A_{/m}$ and $F_{/m}$ of certain prescribed sizes ($\frac{p+1}{3}$ and $1$, respectively), these necessary conditions are also sufficient. Using these facts in conjunction, B-L showed that
    Our motivating point is BL's following result:
    \begin{proposition*}[BL]
        For $C=2\BN\cup B\cup F$, where $B\subset 2\BZ_{-}$ and $F\subset 2\BZ+1$, $C$ arises as a MAC if and only if 
        \[2\BZ\sm(2\BN\cup B)=W+F\]
        for some $W\subseteq\BZ$.
    \end{proposition*}
    
    % In conclusion, B-L has given necessary and sufficient conditions for eventually periodic sets of prime period congruent to 2 mod 3 and with $|F_{/m}|=1$ and $|A_{/m}|=\frac{p+1}{3}$ to arise as MACs. They also simplified these conditions to something very concise in the case of $m=2$. One may then ask if there are other circumstances in which an eventually periodic set $C$ arises as a MAC. We will attempt to treat this direction in this paper.

\subsection{The Setting}
    In this section we will describe the ``setup'' we will be operating under, in the hopes of providing a clearer and more visual picture of the problem at hand. 
    
    In dealing with eventually periodic sets of period $m$, we shall think of them as follows: take the infinite strip in the lattice $\BZ^2$ given by
    \[Z_m=\{(x,y)\in\BZ^2:0\le x\le m-1\}\]
    and consider it as a copy of $\BZ$ by taking
    \[n=x+my.\]
    (For the sake of visual appeal, rather than place these ``dots'' $(x,y)$ on the grid lines, we shall shift the grid lines to the west and to the south by $0.5$, so that the dots lie in the middle of ``boxes''; that is, rather than placing the dots ``Go-style'', we will be placing them ``English chess-style''.) For example, the set $4\BN\cup \{-8,-12\}\cup\{3,6\}$ would look like:
    \begin{center}
       \includegraphics[width=100pt]{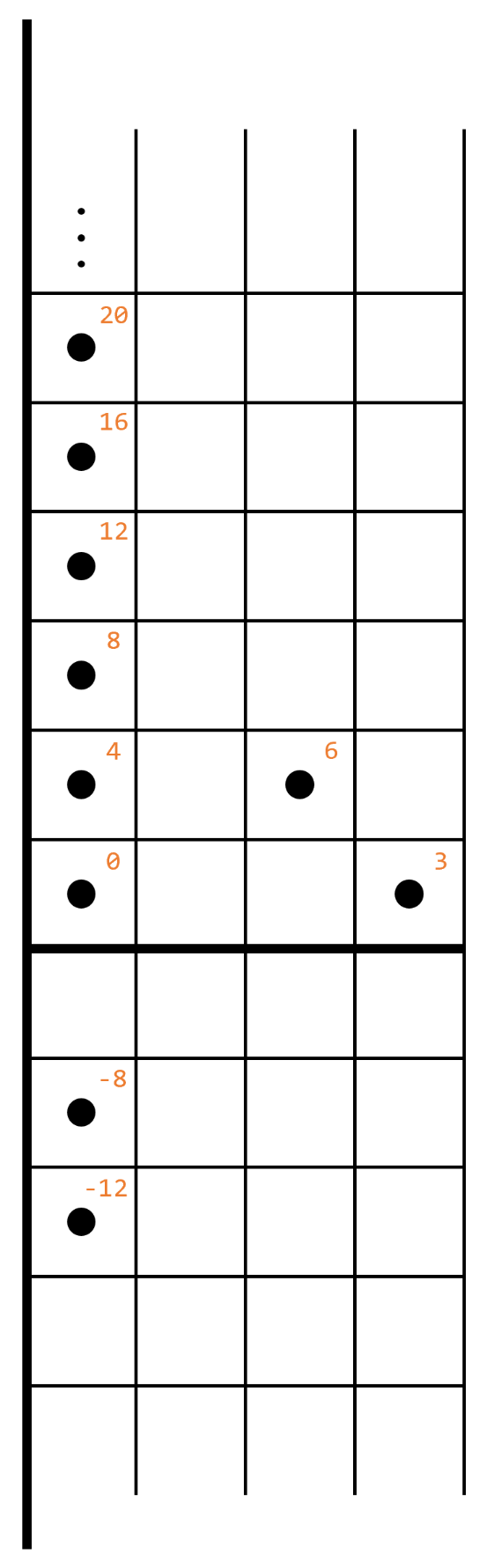}
    \end{center}
    % Similarly we might schematically denote a set e.g. $C=m\BN\cup B\cup F$ by
    % \begin{center}
    %     \textcolor{red}{include graphic here}
    % \end{center}
    
    In order to preserve structure, perhaps it is better to think of this as
    \[Z_m=\quot{\BZ^2}{\pr*{(x,y)\sim (x',y')\te{ if } x\equiv x'\ \rmod{m},\ y'+\frac{x'-x}{m}=y}},\]
    or more concisely %(here the isomorphism is as $\BZ$-modules, hence the symbol $\sf{Mod}_\BZ$)
    % \[Z_m\overset{\sf{Mod}_\BZ}{\cong}\mnfrac{\BZ^2}{\spn_\BZ\{(m,-1)\}}\cong \BZ\]
    \[Z_m\overset{\sf{Mod}_\BZ}{\cong}\mnfrac{\BZ^2}{\BZ\wangle{(m,-1)}}\cong \BZ\]
    where the congruence is that of abelian groups (hence the symbol $\overset{\sf{Mod}_\BZ}{\cong}$). In fact we can take this to be definition:
    \begin{definition}
        Let $Z_m$ be
        \[Z_m\coloneqq \mnfrac{\BZ^2}{\BZ\wangle{(m,-1)}}\cong\BZ.\]
        We will denote the projection map by
        \[\pi_m\colon \BZ^2\lsurj Z_m.\]
        
        Under the identification (which we call the \textit{strip construction})
        \[Z_m=\{(x,y)\in\BZ^2:0\le x\le m-1\},\]
        for $i\in[0,m-1]$ we will call each $\{(x,y)\in Z_m:x=i\}$ the \textit{$i$-th column}, denoted $\Col(i)$ (or $\Col_m(i)$ when there could be confusion with the following definition). We will write $\Col^+(i)$ to denote the subset of $\Col(i)$ with $y\ge 0$, i.e. nonnegative $y$-coordinate. Similarly $\Col^-(i)$ refers to $y<0$.
        
        Similarly we may refer to $\{(x,y)\in \BZ^2:x=i\}$ as the \textit{$i$-th column} of $\BZ^2$, denoted $\Col(i)$. We will write $\Col^+(i)$ to denote the subset of $\Col(i)$ with $y\ge 0$, i.e. nonnegative $y$-coordinate. Similarly $\Col^-(i)$ refers to $y<0$. For a subset $S$ of $\BZ^2$, we also denote the set of columns in which $S$ has elements by $\Col(S)$, with $\Col^+(S)$ defined by $\Col^+(S)\coloneqq \bigcup_{s\in S}\Col^+(s)$ and similarly for $\Col^-(S)$.
        
        \textit{Abuse of notation}: for any subset $S$ of $\BZ$, we will use the same symbol $S$ to denote its isomorphic image in $Z_m\cong \BZ$.
    \end{definition}
    A more pictorial/``topological'' way to think of this is to take $\BZ^2$ and wrap it around horizontally to create a cylinder in a slanted manner, such that each $(mk,y)$ gets glued to $(m(k-1),y+1)$. 
    
    In this setup, the question of whether or not $C$ is an additive complement can be rephrased as whether or not there exists a set of ``translations'' (more precisely this is translations inside $Z_m=\BZ^2/\BZ\wangle{(m,-1)}$) by $W$ such that the union of all such translations of $C$ covers all of $Z_m$. For example, $W=\{1\}$ is a simple shift to the right by one unit. Whether or not $C$ is a \textit{minimal} additive complement can be rephrased as whether or not such a $W$ exists such that every element $c\in C$ has a \textit{dependent element} in the integers $n\in \BZ$, which are defined as follows:
    \begin{definition}
        For a fixed additive complement $W$ of $C$, an element $c\in C$ is said to \textit{have a dependent element} if there is some $\Delta(c)\in\BZ$ such that $\Delta(c)\not\in C\sm\{c\}+W$, i.e. if $c$ is removed then $\Delta(c)$ fails to be covered.
        
        An element of $C$ is said to be a \textit{guardian} if it has a dependent element.
    \end{definition}
    \vspace{-1em}
    % \underline{Important}: As noted by B-L, it is clear that minimality is equivalent to every element of $C$ having a dependent element, i.e. in the union of all translations, $\bigcup_{w\in W}(C+\{w\})=C+W$, every element $c\in C$ has a translate which is covered exactly once.
    \begin{observation*}
    As noted by B-L, it is clear that minimality is equivalent to every element of $C$ having a dependent element, i.e. in the union of all translations, $\bigcup_{w\in W}(C+\{w\})=C+W$, every element $c\in C$ has a translate which is covered exactly once.
    \end{observation*}

\section{Results and Discussions Thereof}
    % Before launching into a lengthy discussion, we should perhaps say that the main results are Theorems 4 and 7. 

B-L described when sets of form $2\BN\cup B\cup F$ arise as MACs. In attempting to generalize this to general $m$, we can restrict our attention to either $F$ a singleton or $B$ an empty set. In the former case, %we shall see that, as a corollary of a more general theorem to be stated later, 
    \begin{proposition}
        For $m\ge 2$, any eventually periodic set
        \[C=m\BN\cup B\cup\{f\}\]
        arises as a MAC. This holds even if $B$ is infinite, unless both $m=2$ and $m\BN\cup B=m\BZ$ are true, in which case it is $m\BZ\cup\{f\}$ does not arise as a MAC. 
    \end{proposition}
In the latter case, %also as a corollary of a more general theorem to be stated later, we have
    \begin{proposition}
        For $|F_{/m}|=1$, any
        \[C=m\BN\cup F\]
        arises as a MAC. 
    \end{proposition}
    Note that the hypothesis $|F_{/m}|=1$ in particular implies $F\neq\emptyset$ is nonempty.
    
    The above two propositions can both be seen as specific instances of the following:
    \begin{proposition}
        For $|F_{/m}|=1$, 
        \[\existss W\subseteq\BZ :m\BZ_{-}\sm B=F+W\implies m\BN\cup B\cup F\te{ is a MAC}.\]
        In words, the existence of a subset $W\subseteq\BZ$ such that $m\BZ_-\sm B=F+W$ implies that $m\BN\cup B\cup F$ is a MAC.
    \end{proposition}  
    % \vspace{-1em}
    \begin{proof}[Proposition 3 implies Proposition 1] Let $B$ be finite. Indeed, in Proposition 1 $F=\{f\}$ is a singleton, which can cover any subset of the integers, and therefore in particular any $B$ has $m\BZ_-\sm B=\{f\}+W$, which by Proposition 3 implies $C=m\BN\cup B\cup\{f\}$ is a MAC. 
    
    For the case where $B$ is infinite, see the Appendix.
    \end{proof}
    % As a remark, Proposition 1 holds even if $B$ is infinite, unless both $m=2$ and $m\BN\cup B=m\BZ$. 
    % The reader will note that this proof does not cover the case where $B$ is infinite, as the given statement of Proposition 3 does not cover the case where $B$ is infinite. We will give a proof of this later. 
    
    \begin{proof}[Proposition 3 implies Proposition 2] Similarly in Proposition 2 we have $B=\emptyset$, so that $m\BZ_-\sm B=m\BZ_-$, which is coverable by any $F$; indeed, just take $W=\{-km-f_1:k\in\BZ_+\}$, where $f_1$ is the maximal element of $F$. 
    % More pictorially, this can be seen by
    % \begin{center}
    %     \textcolor{red}{include graphic here}
    % \end{center}
    (Indeed, it is easy to see that, generally, $\BN$ is coverable by any finite set $F\subseteq\BZ$; then the above is a specific instance, since $m\BZ_-$ is isomorphic to $\BN$ as monoids.) Hence an empty $B$ satisfies the conditions of Proposition 3, which gives Proposition 2. \end{proof}
    
    In fact, Proposition 3 is also a specific case of a more general statement:
    \begin{theorem}
        Let $|B_{/m}|=|F_{/m}|=|A|=1$ with $F_{/m}=\{\wt f\}$, where without loss of generality let $A=\{0\}$. Then the existence of a set cover $\{S_i\}$,
        \[S_1\cup\cdots\cup S_n=m\BN\cup B,\]
        such that each member $S_i$ has
        \[m\BZ\sm S_i=F+W_i\]
        for some $W_i\subseteq\BZ$, implies that if 
        \[m\ge \wt f+2\wt f\floor*{\frac{n}{\wt f}}+\on{mod}_{\wt f}n,\]
        then \[C=m\BN\cup B\cup F\] arises as a MAC.
    \end{theorem}
    % \vspace{-1em}
    \begin{proof}[Theorem 4 implies Proposition 3]
    First let $\wt f>1$. Indeed, Proposition 3 is the case when $m\BN\cup B$ has a cover which consists of a single set, namely $m\BN\cup B$ itself; in this case $n=1$, and the relevant bound is $m\ge \wt f+1$, which is of course always true. This recovers Proposition 3. 
    
    If $\wt f=1$, the argument in the previous paragraph will give that we are done if $m\ge 3$. The $m=1$ case is of course impossible since then $F$ would collide with the column containing the arithmetic progression. 
    
    Hence it remains to prove the case $\wt f=1$ and $m=2$. We claim that $C+(W\cup\{0\})=\BZ$ realizes $C$ as a MAC. Firstly, since $m\BZ_-\sm B$ is infinite in the negative direction, it contains an isomorphic copy of $\BN$ (as monoids), and so by the BL Lemma we can assume that $F$ is minimal with respect to the condition $F+W=m\BZ_-\sm B$. Secondly, since $F+W=m\BZ_-\sm B\subseteq 2\BZ$ lies in the 0-th column, we have that $0\not\in W$. Thirdly, since $F+W=m\BZ_-\sm B$ where $m\BZ_-\sm B$ is unbounded in the negative direction, we have that the sum $C+(W\cup\{0\})=(C+W)\cup (C+\{0\})$ contains infinitely many translates of the arithmetic progression $2\BN$ in the sum $C+W$. These translates are of the form $2\BN+k$, where $k\cong 1\pod 2$, and since there are infinitely many of them, they cover all of $2\BZ+1$. Furthermore, $C+W$ contains $F+W$ and therefore $m\BZ_-\sm B$. On the other half, $C+\{0\}$ contains $m\BN\cup B$. Hence, together, we have $C+(W\cup\{0\})$ contains $m\BN\cup B$, $m\BZ_-\sm B$, and $2\BZ+1$; hence $C+(W\cup\{0\})$ gives all of $\BZ$. Moreover, $C$ is a MAC with respect to $W\cup\{0\}$ since the removal of any element in $2\BN\cup B$ would lead to $C+\{0\}$ not containing that element, and the removal of any element in $F$ would lead to $F+W$ not giving all of $m\BZ_-\sm B$. 
    \end{proof}
    
    Let us remark that, in the worst case scenario, one can always take the set cover to be
    \[m\BN\cup\{b_1\}\cup\cdots\{b_k\}=m\BN\cup B\]
    in the above, so that $n=|B|+1$. This set cover satisfies the required conditions since, as noted earlier, $m\BZ\sm m\BN$ is always coverable by $F$, and $m\BZ\sm\{b\}$ is similarly coverable since it consists of two isomorphic (as monoids) copies of $\BN$, which we have established is coverable by finite sets. In particular, this means that
    \begin{proposition}
        Let $|B_{/m}|=|F_{/m}|=|A|=1$ with $F_{/m}=\{\wt f\}$, where WLOG let $A=\{0\}$. Then 
        \[m\ge \wt f+2\wt f\floor*{\frac{|B|+1}{\wt f}}+\rmod_{\wt f}(|B|+1)\]
        implies that \[C=m\BN\cup B\cup F\] arises as a MAC.
        
        In some sense, this is saying ``any $C=m\BN\cup B\cup F$ (where $|F_{/m}|=1$) is a MAC for sufficiently large $m$'', or ``any $C=m\BN\cup B\cup F$ ($|F_{/m}|=1$) is eventually a MAC''. 
    \end{proposition}
    % \vspace{-1em}
    \begin{proof}[Theorem 4 implies Proposition 5]
         See paragraph preceding Proposition 5.
    \end{proof}
    
    This idea of ``eventually being a MAC'' is one we will explore more and make more precise presently. 
    \begin{definition}
    % Given a any set of points $\wh C$ in the half-lattice $\{(x,y)\in \BZ^2:x\ge 0\}$, 
    Given any set of points $\wh C$ in the lattice $\BZ$, we may consider its image under the projection and then the isomorphism 
    \[\vphi\circ\pi_m\colon \BZ^2\lsurj Z_m\overset{\sim}{\lto}\BZ,\]
    which is a composition we will call $\pi_m$ by abuse of notation. Let this image be denoted $C=\pi_m(\wh C)$; then we may ask whether or not $C$ is a MAC inside $\BZ$. In fact, we can ask this question for varying $m$.
    
    In general, we will call such a $\wh C$ (respectively $\wh A,\ \wh B,\ \wh F$) a \textit{pattern} for/of $C$ (respectively $A,\ B,\ F$), and $\wh C_/$ is defined as \[\wh C_/\coloneqq \{x:(x,y)\in \wh C\te{ for some }y\}\] the set of $x$-coordinates of $\wh C$.
    \end{definition}
    This construction allows us to turn subsets of $\BZ^2$ into subsets of $\BZ$. When such a $\wh C$ has columns which are either consisted of finitely many points (giving $F$) or consisted of finite many points union a ``infinite ray'' of points going to the north (giving $B\cup (m\BN+A)$), this construction gives $C$ an eventually periodic set. 
    %; such a $\wh C$ might be look like
    % \begin{center}
    %     \textcolor{red}{include graphic here}
    % \end{center}
    
    Just as a $\wh C$ gives rise to a $C$ via $C=\pi_m(\wh C)$, given a set $C\subseteq Z_m\cong\BZ$ we may also consider its ``natural'' preimage, denoted $\pi_m^{-1}(C)$, which is the unique preimage $\wh C$ satisfying $\wh C_/\subseteq [m]-1=\{0,\cdotsc,m-1\}$ %(here $\wh C_/$ is defined as $\wh C_/=\{x:(x,y)\in \wh C\te{ for some }y\}$ the set of $x$-coordinates of $\wh C$).
    % For the sake of clarity, WLOG we may let $\wh C$ 
    
    % Let us just remark that in the case $n=1$, i.e. in the case that $m\BZ_-\sm B=F+W$ for some $W$, the relevant bound is $m\ge f+1$, which is of course always true; hence when $m\BZ_-\sm B=F+W$ it is always the case that $m\BN\cup B\cup F$ arises as MAC (for $|B_{/m}|=|F_{/m}|=1$). In particular, if $F$ is a singleton, anything can be covered with $F$, so in particular $m\BZ_-\sm B=\{f\}+W$ always holds, and so any $m\BN\cup B\cup \{f\}$ arises as a MAC.
    
    The next theorem will tell us that, in some sense, all patterns for eventually periodic sets are eventually MACs. That is, we take our eventually periodic set $C$ of period $m$, consider its preimage $\pi_m^{-1}(C)=\wh C$ under the construction above, and consider, for large growing $M$, $\pi_M(\pi_m^{-1}(C))$; the statement is that this set is a MAC in $\BZ$ for all sufficiently large $M$. But before stating this theorem let us define a quantity we will use: for $\wh C$ the pattern for $C$, consider $\wh C_/=\{x:(x,y)\in \wh C\te{ for some }y\}$ the set of $x$-coordinates; this $\wh C_/$ will be a set of separated maximal contiguous ``blocks'', i.e.
    \[\wh C_/=\{ c_{1,1},\cdotsc, c_{1,n_1}\}\cup\cdots \cup\{ c_{k,1},\cdotsc, c_{k,n_k}\},\]
    where each $\{ c_{i,1},\cdotsc, c_{i,n_i}\}$ has $c_{i,j}-c_{i,j-1}=1$ (i.e. is an arithmetic progression of common difference one) and $i<j\implies c_{i,\blt}<c_{j,\blt}$. 
    
    Let $\ell$ be the smallest possible length of a consecutive (i.e. an arithmetic progression of common difference one) set of integers formed by horizontal translates of $\wh C_/$, i.e. the minimal possible length of an interval of integers $[a,b]$ such that there exists $W$ such that $\wh C_/+W=[a,b]$. Then this number is at most
    \[\ell\le c_{k,n_k}-c_{1,1}+c_{k,1}-c_{1,n_1}\vrot{\coloneqq}{180}\te{outerrange}(\wh C_/)+\te{innerrange}(\wh C_/)\]
    since the set
    \[W=[0,\te{innerrange}(\wh C_/)-1]=\{0,1,\cdotsc,\te{innerrange}(\wh C_/)-1\}\]
    gives $\wh C_/+W=\{ c_{1,1}, c_{1,1}+1,\cdotsc, c_{k,n_k}+\te{innerrange}(\wh C_/)-1\}$, which is a consecutive run of integers. 
    % Then an upper bound for $\ell$ the minimal possible length of a consecutive set of integers formed by horizontal translates of $\wh C_/$ is
    % \[\ell\le c_{k,n_k}-c_{1,1}+c_{k,1}-c_{1,n_1}\vrot{\coloneqq}{180}\te{outerrange}(\wh C_/)+\te{innerrange}(\wh C_/),\]
    % which can be thought of as an ``outer range'' (i.e. actual range, or $c_{k,n_k}-c_{1,1}$) plus an ``inner range'' ($c_{k,1}-c_{1,n_1}$). Pictorially this ``block'' structure of $\wh C_/$ might be represented as
    % \begin{center}
    %     \textcolor{red}{\textcolor{red}{include graphic here}}
    % \end{center}
    \begin{theorem}
        For any fixed patterns of $\wh A$, $\wh B$, and nonempty $\wh F$, any
        \[C=(m\BN+A)\cup B\cup F\]
        is a MAC for sufficiently large $m$.
        
        In other words, ``all eventually periodic sets are eventually MACs''.
        
        The constructed bound for this is
        \[m\ge (\ell+1)(|A|+|B|+|F_{/}|),\]
        where $\ell$ denotes the minimal length possible of a consecutive block formed by horizontal translates of $\wh C_{/}$, which is further bounded by
        \begin{align*}
            \ell&\le \te{outerrange}(\wh C_/)+\te{innerrange}(\wh C_/).
        \end{align*}
    \end{theorem}
    % Let us now explain this strange upper bound for $\ell$ and the symbols like Rightmostbloc. In the above statement we have made appeal to this quantity $\ell$. Its upper bound is constructed as follows: for $\wh C$ the pattern for $C$, consider $\wh C_/=\{x:(x,y)\in \wh C\te{ for some }y\}$ the set of $x$-coordinates; this $\wh C_/$ will be a set of separated maximal contiguous ``blocs'', i.e.
    % \[\wh C_/=\{\wt c_{1,1},\cdotsc,\wt c_{1,n_1}\}\cup\cdots \cup\{\wt c_{k,1},\cdotsc,\wt c_{k,n_k}\},\]
    % where each $\{\wt c_{i,1},\cdotsc,\wt c_{i,n_i}\}$ has $c_{i,j}-c_{i,j-1}=1$ (i.e. is an arithmetic progression of common difference one) and $i<j\implies c_{i,\blt}<c_{j,\blt}$. Then the upper bound for $\ell$ is as
    % \[\ell\le c_{k,n_k}-c_{1,1}+c_{k,1}-c_{1,n_1},\]
    % which can be thought of as an ``outer range'' (i.e. actual range, or $c_{k,n_k}-c_{1,1}$) plus an ``inner range'' ($c_{k,1}-c_{1,n_1}$). Pictorially this ``bloc'' structure of $\wh C_/$ might be represented as
    % \begin{center}
    %     \textcolor{red}{\textcolor{red}{include graphic here}}
    % \end{center}
    
    Theorem 6 is also a corollary of a more general theorem:
    \begin{theorem}
        Let $\wh C\subset\BZ^2$ be horizontally bounded (i.e. $\wh C_/$ is a bounded set). Let $\wh F=\wh F_{1}\cup\cdots\cup \wh F_{t}$ denote the (nonempty) columns with finitely many points (where each $\wh F_i$ lies in a distinct column $\Col(\wh F_i)$), and let $\wh K=\wh K_1\cup\cdots\cup \wh K_r$ denote the columns with infinitely many points (where again each $\wh K_i$ lies in a distinct column $\Col(\wh K_i)$). For $\wh C_{/}$ the set of $x$-coordinates of $\wh C$, let $\ell$ denote the minimal length possible of a consecutive block formed by horizontal translates of $\wh C_{/}$, which satisfies
        \begin{align*}
            \ell&\le \te{outerrange}(\wh C_/)+\te{innerrange}(\wh C_/).
        \end{align*}
        
        Suppose that for each $\wh K_i$ there exists a set cover $\wh{\cal S}_i$ consisting of
        \[\wh S_{i,1}\cup\cdots\cup \wh S_{i,|\wh{\cal S}_i|}=\wh K_i\]
        such that, for each $i,j$, there is some collection of (possibly empty) $\wh W_{i,j;\mu},\wh U_{i,j;\nu}\subseteq\BZ^2$ such that
        \[\Col(\wh K_i)\sm \wh S_{i,j}=\bigcup_{\mu=1}^{t} \pr*{\wh F_\mu+\wh W_{i,j;\mu}}\cup\bigcup_{\nu=1}^r \pr*{\wh K_\nu+\wh U_{i,j;\nu}}.\]
        Then
        \[m\ge (\ell+1)\pr*{|\wh F_{/}|+\sum_{i=1}^r |\wh{\cal S}_i|}\]
        implies that $C=\pi_m(\wh C)$ is a MAC in $\BZ$.
    \end{theorem}
    % Here $\ell$ is the smallest possible length of a consecutive (i.e. an arithmetic progression of common difference one) set of integers formed by horizontal translates of $\wh C_/$, i.e. the minimal possible length of an interval of integers $[a,b]$ such that there exists $W$ such that $\wh C_/+W=[a,b]$. In particular this number is at most
    % \[\ell\le\te{outerrange}(\wh C_/)+\te{innerrange}(\wh C_/)\]
    % since the set
    % \[W=[0,\te{innerrange}(\wh C_/)-1]=\{0,1,\cdotsc,\te{innerrange}(\wh C_/)-1\}\]
    % gives $\wh C_/+W=\{\wt c_{1,1},\wt c_{1,1}+1,\cdotsc,\wt c_{k,n_k}+\te{innerrange}(\wh C_/)-1\}$, which is a consecutive run of integers. 
    
    \begin{proof}[Theorem 7 implies Theorem 6] 
        To recover Theorem 6 from Theorem 7, take $\wh C=\pi_m^{-1}((m\BN+A)\cup B\cup F)$ lying in the strip $0\le x\le m-1$. Then $\wh C_/$ is bounded. In this case the finite columns of $\wh C$ are $\wh F=\pi_m^{-1}(F)$ and the infinite columns are $\wh K=\pi_m^{-1}((m\BN+A)\cup B)$, which can be written as $\wh K=\pi_m^{-1}(m\BN+A)\cup\pi_m^{-1}(B)=\Col^+(\wh A)\cup \wh B$, where $\Col^+(\wh A)\coloneqq \bigcup_{\wh a\in \wh A}\Col^+(\wh a)$. It should be noted that each specific column of $\wh K$ is of form $\wh K_i=\Col^+(\wh a_i)\cup \wh B_i$, where $\wh a_i$ is an element of $\wh A$ and $\wh B_i$ are the elements of $\wh B$ lying in the same column as $\wh a_i$. $\ell$ is the same as before. 
         
         Now take the set cover of the infinite columns
         \[\wh K=\Col^+(\wh a_1)\cup\cdots\cup \Col^+(\wh a_{|A|})\cup \{\wh b_1\}\cup\cdots\cup\{\wh b_{|B|}\},\]
         i.e.
         \[\wh K_i=\Col^+(\wh a_i)\cup\bigcup_{\wh b\in \wh B_i}\{\wh b\},\]
         where $|\cal S_i|=1+|B_i|$ and $\sum_i |\cal S_i|=|A|+|B|$. This set cover satisfies the hypotheses since firstly
         \[\Col(\wh K_i)\sm \Col^+(\wh a_i)=\Col^-(\wh a_i)=F_j+W_{i,j}\]
         for any $F_j$; to see this recall from earlier that any half column (which is isomorphic as a monoid to $\BN$ if we take addition to only affect the $y$ coordinate) is coverable by any finite set. Similarly, secondly
         \[\Col(\wh K_i)\sm \{\wh b\}=F_j+W_{i,j}\]
         for any finite column $F_j$, which is true since $\Col(\wh K_i)\sm \{\wh b\}$ consists of two infinite rays, one pointing up and one pointing down, and as established earlier these rays, each isomorphic as monoids to $\BN$, are coverable by finite sets.
         
         Then Theorem 7 states that
         \[m\ge (\ell+1)(|F_{/}|+|A|+|B|)\]
         implies that $C$ is a MAC, which is precisely the statement of Theorem 6.
    \end{proof}
    
    The reader might note that Theorems 4 and 7 say very similar things, namely that for large enough $m$ the projection under $\pi_m$ of some pattern will be a MAC in $\BZ$. However, Theorem 4 is not a corollary of Theorem 7 due to the bounds. Indeed, applying Theorem 7 to the setting of Theorem 4 will yield only
    \[m\ge (2\wt f+1)(n+1)\]
    where $n$ is the size of the set cover $\cal S$, which is much worse than the bound given in Theorem 4. As the reader will see in the proofs in the following section, the construction for Theorem 4 feels ``tight'' or ``efficient'' in some sense while in Theorem 7 we are much sloppier. This is perhaps to be expected; Theorem 4 deals with a very specific type of set (namely $|A_{/m}|=|B|=|F_{/m}|=1$), while Theorem 7 deals with a much broader class, so one might expect that it is easier to derive better bounds in the former case than the latter. 
    
    We should also note that, by choosing appropriate $K$, Theorem 7 answers the question of Burcroff and Luntzlara mentioned at the end of the Introduction in a large class of cases. Indeed, writing $C_i=(m\BN+A_i)\cup B_i\cup F_i$ and $A_i=\bigcup_j A_{i,j}$ and $B_i=\bigcup_j B_{i,j}$ (where $A_{i,j},B_{i,j}$ lies in a single column), by taking $K$ to be $K=(m\BN+A_1)\cup B_1 \cup (-(m\BN+A_2))\cup (-B_2)\subseteq C_1\cup(-C_2)$, Theorem 7 tells us that whenever we can partition $B_{i,j}$ by $S_{i,j,k}$ such that $\Col^-(\min A_{i,j})\sm S_{i,j,k}$ can be set-covered by appropriate translates of the different columns in $C_1,C_2$, if $m$ is larger than some bound depending on the size of our cover, the number of equivalence classes mod $m$ represented by $F_1\cup (-F_2)$, and the horizontal distribution (when we draw it in $Z_m$ form) of $C_1$ and $C_2$, then $C_1\cup (-C_2)$ is a MAC. As an example, in the case that $B_1=B_2=\emptyset$ and at least one of $F_1,F_2$ is nonempty, since translations of finite sets can cover $\BN$, we obtain the bound that 
    \[m\ge (\ell+1)(|(F_1)_/|+|(F_2)_/|+|(A_1)_/|+|(A_2)_/|)\]
    implies $C_1\cup(-C_2)$ arises as a MAC. 
    
    As another example of Theorem 7, we could consider the case $K=m\BN\cup B$ where $B$ is infinite and $F=\{f\}$ is a singleton. Then we can take the set cover to have one set, namely $K$ itself, for $m\BZ\sm(m\BN\cup B)=m\BZ_-\sm B$ is coverable by (translates of) the singleton $\{f\}$. Furthermore, in this case $\ell=2\wt f$. Then the bound from Theorem 7 tells us that
    \[m\ge (2\wt f+1)2\]
    implies that $C=m\BN\cup B\cup \{f\}$ is a MAC. It turns out this is true for smaller $m$ as well, as long as not both of $m=2$ and $m\BN\cup B=m\BZ$ are true.
    
    % \textcolor{orange}{include a few paragraphs/examples about tightness}
    
    Having discussed at length the results, it remains to prove Theorems 4 and 7 (as we have noted in the discussions above, all other results are actually corollaries of these two).

    %%%%%%% P R O O F   T I M E %%%%%%%

\section{A Formalism of Formal Power Series}
    Before giving the proofs of our main theorems, namely Theorems 4 and 7, we will develop a language of ``formal power series'' in which the proofs are much easier to relate. We should however give a disclaimer beforehand that these formal power series do not possess the \textit{soul} of the technique of generating functions, which is namely the idea of collapsing long expressions into short ones or vice versa (e.g. the identity $\sum_n x^n=\frac{1}{1-x}$) in order to achieve clever manipulations. The formal power series we introduce here will not engage in such acrobatics and will instead serve solely as bookkeepers. %We apologize in advance to generating functionologists (a group to which the author hopes to belong) who may have been hoping for something exciting. 
    
    The point of this is to make it easier to show that $C$ is a MAC, given a claimed complement $W$. Roughly, this formalism will turn a set $S\subseteq\BZ$ into a formal power series.
    
        In our strip construction $Z_m$, each column $\Col(i)=\{(x,y)\in Z_m:x=i\}$ can be thought of as a copy of $\BZ$ with the obvious addition structure (add the $y$-coordinates). We will denote%\footnote{We apologize for the similar notations $Z_m$ and $Z^m$ here, but hopefully the superscript on the latter is reminiscent of the vertical column structure.} this by
    \begin{definition}
        Endow $\Col(i)=\{(x,y):x=i\}\subset Z_m$ with an isomorphism
        \[\Col(i)\cong \BZ\]
        where we take addition in $\Col(i)$ to be addition of the $y$-coordinates. This isomorphism is such that $(i,0)\in \Col(i)$ corresponds to $0\in \BZ$.
        
        % As these are all isomorphic, we will write $Z^m$ for any $Z^m(i)$.
        
        For a set $S\subseteq \Col(i)\subset Z_m$ which lies entirely in a single column, we will let $\ol S\subseteq \BZ$ denote its image in $\BZ$ under the isomorphism $Z_m\supset \Col(i)\cong\BZ$. 
        
        In the backwards direction, given a set $S\subseteq\BZ$, we will let $\wc S(i)\subseteq \Col(i)\subset Z_m$ denote its image under the inverse isomorphism. 
    \end{definition}
    The above is nothing more than saying that the set of all integers in a single congruence class mod $m$ forms a copy of $\BZ$.
    
    % We will use this idea to develop a ``calculus'' to work in. The point of this is to make it easier to show that $C$ is a MAC, given a claimed complement $W$. Roughly, this formalism will turn a set $S\subseteq\BZ$ into a formal power series.
    %\footnote{Perhaps we should say that mathcal $\cal Q$ is chosen here to stand for ``ensemble'', or to stand for the exponential symbols.}\footnote{It should be noted that this $q^A$ symbolic notation is inspired by the similar notation from the characters of representation theory, e.g. in the Weyl character formula.}
    
    But before we can describe how our formalism will turn the data of a set $S\subseteq\BZ$ into an object, we must first describe in what world this object will live. In the following definition the symbol $\sqcup$ refers to the disjoint union, which keeps track of multiplicities, and the symbol $\oplus$ refers to the Minkowski sum of sets $+$ except with multiplicities taken into account\footnote{For example, $\{0,1\}+\{1,2\}=\{1,2,3\}$ whereas $\{0,1\}\oplus\{1,2\}=\{1,2,2,3\}$.}, i.e. a ``disjoint'' Minkowski sum.
    \begin{definition}
        Consider the $\BZ$-algebra generated by symbols of form $q^A$ for $A\in\BN^\BZ$ (here $\BN^\BZ$ refers to subsets of $\BZ$ with multiplicity allowed, i.e. a ``weighted subset'' with weights, which encode multiplicity, in $\BN$), modded out by relations $q^{\emptyset}=0$, $q^{\{0\}}=1$, $q^Aq^B-q^{A\oplus B}=0$, and $q^A+q^B-q^{A\sqcup B}=0$; in symbols this is
        \[\cal Q\coloneqq \quot{\BZ\big[\{q^A:A\in\BN^\BZ\}\big]}{\set{q^\emptyset=0,\ q^{\{0\}}=1,\ q^Aq^B=q^{A\oplus B},\ q^{A}+q^B=q^{A\sqcup B}}}.\]
        
        We can then consider the following polynomial ring over this algebra:
        \[\Xi_m\coloneqq \mnfrac{\cal Q[[x]]}{\wangle{x^m-q^{\{1\}}}}.\]
        
        When $q^B=q^A+q^C$ for some set $C\in\BN^\BZ$, we will say $q^B\ge q^A$. 
    \end{definition}
    \textit{Abuse of notation}: For a singleton $S=\{n\}$, instead of writing the cumbersome $q^S=q^{\{n\}}$, we will write $q^n$. For example we will write $q^1$ in place of $q^{\{1\}}$. Similarly we will write $q^0=1$ instead of $q^{\{0\}}$. We will also later drop the notation $\oplus$ and only use $+$, relying on context\footnote{Perhaps we should note that $A+B\subseteq A\oplus B$; indeed, once we remove multiplicity, these two sets are the same.} for whether we consider multiplicity or not. Generally speaking, whenever we are in the context of these exponential symbols, or later in the context of formal power series, the symbol $+$ will be taken to mean with multiplicity. 
    
    Our choice of notation $\BN^\BZ$ here is in line with the notation $\{0,1\}^\BZ$ for the power set $\cal P(\BZ)$. Hence regular subsets of $\BZ$ are those members of $\BN^\BZ$ whose weights (i.e. multiplicities) are either 0 or 1, so that every member of $\cal P(\BZ)$ is a member of $\BN^\BZ$.
     
    The reason why we take this ideal to quotient by in the definition of $\Xi_m$ will be clear later. It is in this ring $\Xi_m=\mnfrac{\cal Q[[x]]}{\wangle{x^m-q^{1}}}$ that our formal power series shall live.
    
    These symbols, appropriately, behave like exponentials and correspond to
        \begin{align*}
            \emptyset&\llra q^{\emptyset}=0,\\
            \{0\}&\llra q^0=1,\\
            A&\llra q^A,\\
            A\oplus B&\longleftrightarrow q^Aq^B=q^{A\oplus B},\\
            A\sqcup B&\llra q^A+q^B,\\
            B\sm A&\llra q^B-q^A,
            \intertext{and the distributive property of multiplication corresponds to }
            A\oplus (B\sqcup C)=(A\oplus B)\sqcup (A\oplus C)&\llra q^A(q^B+q^C)=q^Aq^B+q^Aq^C=q^{A\oplus B}+q^{A\oplus C}.
            \intertext{Defining $B\ominus A$ to be the set $C$ such that $A\oplus C=B$ (if it exists), we also have}
            B\ominus A&\llra \frac{q^B}{q^A}.
        \end{align*}
    Note that by extending the notion of setminus $B\sm A$ to include cases when $A$ is not necessarily a subset of $B$, we can make sense of expressions such as $-q^A$. Indeed, treating $\BN^\BZ$ as a semi-ring with addition corresponding to $\sqcup$ and multiplication corresponding to the disjoint Minkowski sum $\oplus $, we can complete this to a ring by introducing symbols of form $B\sqminus A\llra q^B-q^A$, appropriately quotienting so that $B\sqminus A=D\sqminus C\iff B\sqcup C=D\sqcup A$, i.e. $q^B-q^A=q^D-q^C\iff q^B+q^C=q^D+q^A$.

    The idea of this formalism is to do the following: 
    \begin{definition}
    Given a set $S\subseteq Z_m\cong\BZ$, write $S=S_1\sqcup\cdots\sqcup S_n$ where each $S_i$ lies in a single column labeled by distinct $\wt s_i\in\BZ/m\BZ$. Then the information in $S$ is the same as the data
    \[C\longleftrightarrow \Big\{ (\ol S_1,\wt s_1),\cdotsc,(\ol S_n,\wt s_n)\Big\},\]
    where the first entry in each pair indicates the ``shape'' of the elements of $S$ in the column labeled by the second entry of each pair. We call this the \textit{shape list form}. We also overload the symbol $\sqcup$ and write this as
    \[C=(\ol S_1,\wt s_1)\sqcup\cdots\sqcup (\ol S_n,\wt s_n)=\bigsqcup_{i=1}^n (\ol S_i,\wt s_i).\]
    \end{definition}
    
    Our formalism will now take this data and put it into a power series in the following way. Define:\nc{\Rng}{\on{Range}}
    \begin{definition}
        Given a set $S=S_1\sqcup \cdots\sqcup S_n\subseteq Z_m$, where each $S_i$ lies in a single column, let us write 
        \[S(x)\coloneqq \sum_{i=1}^n q^{\ol S_i} x^{\wt s_i}\overset{\te{abuse}}{\coloneqq} \sum_{i=1}^n q^{ S_i} x^{ s_i}\in \Xi_m,\]
        where in an abuse of notation we have written $x^{s_i}$ instead of $x^{\wt s_i}$ and $q^{S_i}$ instead of $q^{\ol S_i}$ for the sake of simplicity. We write $[x^i]S(x)$ for the coefficient in front of the $x^i$ term (after reducing mod $x^m-q^1$ until all powers are less than or equal to $m-1$).
        
        Then it is obvious that, for $A,B\subseteq Z_m$, we have
        \begin{align*}
            A(x)+B(x)&=(A\sqcup B)(x),\\
            A(x)B(x)&=(A+B)(x).
        \end{align*}
        
        Note well that such formal power series are in bijection with weighted (i.e. we allow multiplicities for each element) subsets of $Z_m$.
        
        For a power series $S(x)\in\Xi_m$ presented in a form such that the exponents appearing in $S(x)$ are in the range $[0,m-1]$, let us denote %(Rng here stands for Range)
        \[\Rng(S(x))\coloneqq \{\te{exponents appearing in $S(x)$}\}=S_{/m}.\]
    \end{definition}

    Now we should also explain why in defining $\Xi_m$ we are quotienting out by the ideal $\wangle{x^m-q^1}$. This is simply because our sets live in $Z_m$. For example, for $A=\{1\}$, we have $A(x)=q^1 x^0$, which is the same as $A(x)=q^0 x^m$ since in the latter description $q^S\coloneqq q^{\ol S}=q^0$ corresponds to $S=\{m\}=\{(m,0)\}\subset Z_m$, which is the same as $\{(m,0)\}=\{(0,1)\}\subset Z_m$ due to the definition of $Z_m$.
    
    In claiming that $(A+B)(x)=A(x)B(x)$ and $(A\sqcup B)(x)=A(x)+B(x)$, we have omitted some minor checks; these are covered in the Appendix at the end of the paper.

    The key in this definition is that now, when claiming that $C$ is a MAC to $W$, rather than compute $C+W$ and show it is $\BZ$, meanwhile proving that all elements of $C$ have dependents, we can instead take their formal power series, multiply, and check the coefficient in front of each $x^i$, which is equivalent to checking column-by-column that $C+W=\BZ$ minimally. As expected, there is no new content in this formalism (it's just notation), but this will make writing down certain proofs much more concise. 
    
    More precisely, the condition of $C+W=\BZ$ is the same as every $i$-th coefficient $[x^i](C+W)(x)$ of $(C+W)(x)$ having $[x^i](C+W)(x)=q^S$ where $S\supseteq\BZ$ (here $i\in[0,m-1]$), and the condition of $C$ being minimal will be checked within each term $q^S x^s$; there, we will check whether or not $S\supseteq \BZ$ contains elements dependent on elements of $C$, and whether or not the union over all such $S$ covers all elements of $C$ (so that every element of $C$ has a dependent element). 
    That is, checking that $C$ is minimal will be equivalent to giving a partition of $C$
    \[\bigcup_i C_i=C\]
    such that
    % \[\foralls i,\ \existss w\in W,\ n\in\BZ/m\BZ:[x^n](C+W)(x)\not\ge 2q^{C_i+w}.\]
    \[\foralls i,\ \existss w\in W,\ n\in\BZ/m\BZ:[x^n](C+W)(x)= q^{C_i+w}+q^S,\ \wt S\cap \wt{(C_i+w)}=\emptyset.\]
    In words, this is saying that for any $C_i$ in this partition, there is some element $w\in W$ and column number $n\in\BZ/m\BZ$ such that the coefficient $[x^n](C+W)(x)$, which counts the results of $C+W$ in the $n$-th column with multiplicity, contains the elements of $C_i+w$ exactly once.

\section{Proofs of Main Results}
    In this section we shall prove the main results, namely Theorems 4 and 7. 
    
    The language of the formal power series will make this process easier to communicate, but for sake of transparency we should say that this is not how we came up with these theorems. Generally, perhaps a reasonable strategy to come up with these statements might be to stare at and play with pictures like the strip construction $Z_m$ (this is what we did), and to write down a proof one would use these formal power series. In writing down the proofs in these sections, we have tried our best to be very explicit and write down all the computational details, even if they are completely straightforward; as a result the proofs are rather long, but we hope that the trade-off is that the readers will be able to read along and confirm that the proofs are correct without having to separately compute/check things on paper themselves.
    
    Perhaps it should be noted that these theorems are much easier to see pictorially than symbolically; as unfortunately is the case at times with mathematics at large, symbols, whilst affording more precision, obscure intuition and the flow of logic.

\subsection{Proving Theorem 4}
We first prove Theorem 4. The idea is to give a construction for $C+W=\BZ$ in which all elements of $C$ have a dependent element; namely, each $S_i\subseteq m\BN\cup B$ will have its dependent elements concentrated in a single column. We shall spread these columns containing dependents out amongst ``cars'' of length $2\wt f$. Here the word ``cars'' refers to the vehicle, in which ``passengers'' (i.e. the columns containing dependents of $C$) will fit, and we will fit all these cars into a giant garage (i.e. $Z_m$), and the idea is that if the garage is long enough (i.e. if $m$ is big enough) then all these cars will fit.

% We shall spread these columns containing dependents out amongst ``magazines''\footnote{This analogy, while evocative, is not entirely accurate, as a quick Google search seems to suggest that there are no firearms which use multiple magazines simultaneously.} of length $2\wt f$. Here the word ``magazines'' refers to the firearm part, in which ``bullets'' (i.e. the columns containing dependents of $C$) will fit, and we will fit all these magazines into a giant gun (i.e. $Z_m$), and the idea is that if the gun is long enough (i.e. if $m$ is big enough) then all these magazines will fit. Pictorially this is straightforward and intuitive to see:
% \begin{center}
%     \textcolor{red}{include graphic here}
% \end{center}

In the proof below, the rough outline will be as follows: we will give a construction of a set $V$ and claim that $C$ is a MAC to $V$; we will calculate the formal power series of these sets; we will multiply the formal power series together; and lastly we will check term-by-term in the power series that $C+V=\BZ$ and that $C$ is minimal with respect to this condition. Since the power series determines the set, we could have just given $V(x)$, but we go the extra step of writing down what $V$ is for this first proof for the sake of transparency. 
\begin{proof}[Proof of Theorem 4]
     For ease of reading we will separate the proof into sections which are italicized. 
     
     \medskip 

     \textit{(1) Observations}. Let $\{S_i\}$ be the set cover in the theorem assumptions, i.e. $S_1\cup\cdots\cup S_n=m\BN\cup B$  such that each member $S_i$ has $m\BZ\sm S_i=F+W_i$ for some $W_i\subseteq\BZ$. Note that the finiteness of $B$ implies $m\BN\cup B$ is bounded below which implies $S_i$ must also be bounded below. As such, $m\BZ\sm S_i$ will contain a shifted copy of $\Col^-(0)$; that is, an infinite ray of integers (more precisely this ray consists of multiples of $m$) starting at $\min(S_i)-m$ and pointing in the negative direction; in particular this infinite ray is 
     \[-m\BN+\min(S_i)-m\subseteq m\BZ\sm S_i;\]
     we can think of this as an isomorphic (as monoids) copy of $\BN$ sitting below $\min(S_i)$. 
     
     By assumption there is some $W_i$ such that $F+W_i=m\BZ\sm S_i=\Col(0)\sm S_i$. Consider a minimal $W_i'\subseteq W_i$ such that $F+W_i'=(\Col^+(0)+y_i)\sm S_i$ for some\footnote{Note that such an $y_i$ is necessarily nonpositive, where we take $\min B=0$ if $B=\emptyset$. Also recall that $\Col^+$ is defined to be inclusive (including 0) whereas $\Col^-$ is defined to be exclusive (not including 0). Note also that this $y_i$ may depend on $i$.} $y_i\in\Col^-(0)+\min B+1$. (This $W_i'$ exists because $F$ is finite, so that if the sum $F+W_i=\Col(0)\sm S_i$ is some set continuing infinitely in the negative direction, then necessarily eventually (in the negative direction) this sum is just consecutive translations of $F$, i.e. there is some $W_i''\subseteq W_i$ such that $F+W_i''=\Col^-(0)+z_i$ for some negative number $z_i$.) Then, noting that 
     \[F+(\Col^-(0)+y_i-\max F)=\Col^-(0)+y_i\cong\BN,\]
     where the latter isomorphism is that of monoids, we may apply the BL Lemma to obtain a modification/replacement\footnote{This notation implies a dependence on $i$, and this dependence comes from $y_i$.} $U_i$ of $(\Col^-(0)+y_i-\max F)$ such that
     \[F+U_i=\Col^-(0)+y_i\]
     with $F$ minimal, i.e. wherein no proper subset $F'\subset F$ satisfies the same equation. Then, with respect to $F+U_i=\Col^-(0)+y_i$, we have that every element of $F$ has a dependent element in $\Col^-(0)+y_i$. Then consider\footnote{Note that the disjoint union here is redundant, as $W_i'$ and $U_i$ are necessarily disjoint because their sums with $F$ are disjoint. } $W_i'\sqcup U_i$; the sum of $F$ with this set is
     \[F+(W_i'\sqcup U_i)=(F+W_i')\sqcup (F+U_i)=\bigpr{(\Col^+(0)+y_i)\sm S_i}\sqcup \bigpr{\Col^-(0)+y_i}=\Col(0)\sm S_i.\]
     
   Hence for the rest of this proof, by redefining the symbol $W_i$ to be
     \[W_i\coloneqq W_i'\sqcup U_i,\]
     we can assume that 
     \[F+W_i=m\BZ\sm S_i,\]
     and $F$ is minimal with respect to this equation; moreover, by construction we can find dependent elements of $F$ which are of form $\BZ\ni \delta< y_i\le \min B$. Note that, passing to $\Col(i)\cong\BZ$, this equation reads
     \[\ol F+\ol W_i=\BZ\sm \ol S_i,\]
     with dependent elements of $F$ which are of form $\ol\delta<\ol y_i\le\min \ol B$. 
    
    %  For ease of notation for the rest of this proof let us redefine the symbol $W_i$ to be
    %  \[W_i\coloneqq W_i'\sqcup U_i.\]
    %  Then
    %  \[F+W_i=m\BZ\sm S_i,\]
    %  and $F$ is minimal with respect to this equation; moreover, by construction we can find dependent elements of $F$ which are of form $\BZ\ni \delta< y_i\le \min B$. Note that, passing to $\Col(i)\cong\BZ$, this equation reads
    %  \[\ol F+\ol W_i=\BZ\sm \ol S_i,\]
    %  with dependent elements of $F$ which are of form $\ol\delta<\ol y_i\le\min \ol B$. 
     
    %  Then, by the BL Lemma, since
    %  \[F+W_i=m\BZ\sm S_i\supseteq (-m\BN+\min(S_i)-m)\cong\BN,\]
    %  there exists a modification\footnote{For full disclosure maybe we should add that, when we define the additive complement of $C$ later, only one of the $W_\te{blah}'$ need to have a ``prime'' symbol, since as we explain later the use of this modified $W_\te{blah}$ is only to create the dependent elements of $F$, of which there is only one column; but we may as well make all of them have the prime mark.} $W_i'$ of $W_i$ such that
    %  \[F+W_i'=F+W_i=m\BZ\sm S_i\]
    %  but for which $F$ is minimal with respect to this condition (i.e. no proper $F'\subset F$ has $F'+W_i'=F+W_i'$). Then, with respect to $F+W_i'=m\BZ\sm S_i$, we have that every element of $F$ has a dependent element in $m\BZ\sm S_i$. For the sake of ease of notation, for the rest of this proof we will recast the variable $W_i$ to
    %  \[W_i\coloneqq W_i',\]
    %  so that $F+W_i=m\BZ\sm S_i$ where $F$ is minimal with respect to this equation. Note that, passing to $Z^m(i)\cong\BZ$, this equation reads
    %  \[\ol F+\ol W_i=\BZ\sm\ol S_i.\]
     
     Let us also note that, since $S_i\subseteq m\BN\cup B$ and therefore $m\BZ\sm S_i$ are concentrated in a single column (i.e. a single congruence class mod $m$) and since $F$ is also concentrated in a single column, $F+W_i=m\BZ\sm S_i$ implies that $W_i$ is also concentrated in a single column, i.e. all elements of $W_i$ are equivalent mod $m$. In fact we know what this congruence class is; $F+W_i=m\BZ\sm S_i$ under projection implies $\wt f+\wt w_i=0$, i.e.\footnote{Here we have written $\wt w$ instead of $\wt w_i$ since they are all the same anyway.}
     \[\wt w=-\wt f.\]
     
     Similarly, as $F$ is finite and $m\BZ\sm S_i$ (as noted earlier) is unbounded below, in order for $F+W_i=m\BZ\sm S_i$ to be true it must be the case that $W_i$ contains infinitely many negative elements. Therefore, $\ol W_i$ must also contain infinitely many negative elements.
     
     For sake of brevity, let
     \[k=\floor*{\frac{n}{\wt f}},\quad r=\rmod_{\wt f}n.\]
     
     \medskip 
     
     \textit{(2) Construction}. To be upfront, we will immediately give the construction\footnote{Here we use the notation $[a,b]=\{a,a+1,a+2,\cdotsc,b-1,b\}$. } of $V$ to which $C$ shall be a MAC. As you can see it is quite a mess, and for that reason we will not be working with all of $V$ all at once, and will instead cut it up into little pieces and consider one at a time.
     \begin{align*}
        V&=[\wt f,2\wt f-1]\cup (W_1+\wt f)\cup (W_2+\wt f+1)\cup\cdots\cup (W_{\wt f}+2\wt f-1)\\
         &\qquad \cup [3\wt f,4\wt f-1]\cup (W_{\wt f+1}+3\wt f)\cup(W_{\wt f+2}+3\wt f+1)\cup\cdots\cup (W_{2\wt f}+4\wt f-1)\\
         &\qquad\ \vdots\\
         &\qquad \cup [(2k-1)\wt f,2k\wt f-1]\cup(W_{(k-1)\wt f+1}+(2k-1)\wt f)\cup\cdots\cup (W_{k\wt f}+2k\wt f-1)\\
         &\qquad \cup [(2k+1)\wt f,(2k+1)\wt f+r-1]\cup (W_{k\wt f+1}+(2k+1)\wt f)\cup (W_{k\wt f+2}+(2k+1)\wt f+1)\cup\cdots\\
         &\qquad\qquad\cup (W_{k\wt f+r}+(2k+1)\wt f+r-1)\\
         &\qquad\cup \bigpr{m\BZ+2k\wt f+r}\cup\bigpr{m\BZ+2k\wt f+r-1}\cup\cdots\cup \bigpr{m\BZ+(2k+1)\wt f-1}\\
         &\qquad \cup \bigpr{m\BZ+(2k+2)\wt f}\cup\cdots\cup\bigpr{m\BZ+m-1},
     \end{align*}
     More concisely and precisely (i.e. collapsing appropriate terms into unions) (here whenever the upper limit is smaller than the lower limit, e.g. $\cup_{j=1}^0$, this is taken to be the empty set by convention; this convention will make our claims true when $r=0$),
     \begin{align*}
         V&=\bigcup_{i=1}^k\pr*{\big[(2i-1)\wt f,2i\wt f-1\big]\cup\bigcup_{j=1}^{\wt f}\big(W_{(i-1)\wt f+j}+(2i-1)\wt f+j-1\big) }\\
         &\qquad \cup \big[(2k+1)\wt f,(2k+1)\wt f+r-1\big]\cup \bigcup_{j=1}^r\big(W_{k\wt f+j}+(2k+1)\wt f+j-1\big)\\
         &\qquad \cup \bigcup_{j=r+1}^{\wt f}\big(m\BZ+2k\wt f+j-1\big)\\
         &\qquad\cup \bigcup_{j=(2k+2)\wt f}^{m-1}\bigpr{m\BZ+j}.
     \end{align*}
      Note that the translation factors $(2i-1)\wt f+j-1$ in the expression $W_{(i-1)\wt f+j}+(2i-1)\wt f+j-1$ are precisely the members of the intervals of integers in $V$.

      Let us denote the shorthand (here $1\le j\le \wt f$)
      \[W_{i,j}\coloneqq W_{(i-1)\wt f+j},\quad S_{i,j}\coloneqq S_{(i-1)\wt f+j}.\]
      Furthermore define ``car blocs'' $V_i$ of $V$ for $1\le i\le k$ as follows:
      \[V_i\coloneqq \big[(2i-1)\wt f,2i\wt f-1\big]\cup\bigcup_{j=1}^{\wt f}\big(W_{(i-1)\wt f+j}+(2i-1)\wt f+j-1\big).\]
      Extending this notation, let us also define for $i=k+1$ the ``remainder bloc''
      \[V_{k+1}\coloneqq \big[(2k+1)\wt f,(2k+1)\wt f+r-1\big]\cup \bigcup_{j=1}^r\big(W_{k\wt f+j}+(2k+1)\wt f+j-1\big)\]
      as well as the ``filler blocs'' 
      \[V_{\BZ}\coloneqq \bigcup_{j=r+1}^{\wt f}\big(m\BZ+2k\wt f+j-1\big)\cup \bigcup_{j=(2k+2)\wt f}^{m-1}\bigpr{m\BZ+j}.\]
      In shape list form these sets are
      \begin{align*}
          V_i&=\big(\ol W_{i,1},(2i-2)\wt f\big)\sqcup\cdots\sqcup\big(\ol W_{i,\wt f},(2i-1)\wt f-1\big)\sqcup\big(\{0\},(2i-1)\wt f\big)\sqcup\cdots\sqcup\big(\{0\},2i\wt f-1\big),\\
          V_{k+1}&=\big(\ol W_{k+1,1},2k\wt f\big)\sqcup\cdots\sqcup\big(\ol W_{k+1,r},2k\wt f+r-1\big)\sqcup \big(\{0\},(2k+1)\wt f\big)\sqcup \cdots\sqcup\big(\{0\},(2k+1)\wt f+r-1\big),\\
          V_\BZ&=\bigpr{\BZ,2k\wt f+r}\sqcup\cdots\sqcup\bigpr{\BZ,(2k+1)\wt f-1}\sqcup \bigpr{\BZ,(2k+2)\wt f}\sqcup\cdots\sqcup \bigpr{\BZ,m-1},
      \end{align*}
      which in a more condensed form is
      \begin{align*}
          V_i&=\bigsqcup_{j=1}^{\wt f}(\ol W_{i,j},(2i-2)\wt f+j-1)\sqcup\bigsqcup_{j=1}^{\wt f}(\{0\},(2i-1)\wt f+j-1),\\
          V_{k+1}&=\bigsqcup_{j=1}^{r}(\ol W_{k+1,j},2k\wt f+j-1)\sqcup \bigsqcup_{j=1}^r(\{0\},(2k+1)+j-1),\\
          V_\BZ&=\bigsqcup_{j=r+1}^{\wt f}(\BZ,2k\wt f+j-1)\sqcup\bigsqcup_{j=(2k+2)\wt f}^{m-1}(\BZ,j).
      \end{align*}
      
      Note wellthat the unions in our definition of $V$ can actually be taken to be disjoint unions (since none of the sets intersect)
      \[V=V_1\sqcup\cdots\sqcup V_k\sqcup V_{k+1}\sqcup V_{\BZ}.\]
      
      The claim then is that $C$ is a MAC of $V$.
      
      \medskip 
      
      \textit{(3) Generatingfunctionology}. Reading from the shape list forms, we may readily see that the formal power series for these blocs are 
      \begin{align*}
          V_i(x)&=q^{\ol W_{i,1}}x^{(2i-2)\wt f}+\cdots+q^{\ol W_{i,\wt f}}x^{(2i-1)\wt f-1}+q^{0}x^{(2i-1)\wt f}+\cdots+q^{0}x^{2i\wt f-1}\\
          &=\sum_{j=1}^{\wt f}\bigpr{q^{\ol W_{i,j}}x^{(2i-2)\wt f+j-1}+x^{(2i-1)\wt f+j-1}};\\
          V_{k+1}(x)&=q^{\ol W_{k+1,1}}x^{2k\wt f}+\cdots+q^{\ol W_{k+1,r}}x^{2k\wt f+r-1}+q^{0}x^{(2k+1)\wt f}+\cdots+q^{0}x^{(2k+1)\wt f+r-1}\\
          &=\sum_{j=1}^r\bigpr{q^{\ol W_{k+1,j}}x^{2k\wt f+j-1}+x^{(2k+1)\wt f+j-1}};\\
          V_\BZ(x)&=q^\BZ x^{2k\wt f+r}+\cdots+q^\BZ x^{(2k+1)\wt f-1}+q^\BZ x^{(2k+2)\wt f}+\cdots+q^\BZ x^{m-1}\\
          &=\sum_{j=r+1}^{\wt f} q^\BZ x^{2k\wt f+j-1}+\sum_{j=(2k+2)\wt f}^{m-1}q^\BZ x^j.
      \end{align*}
      
      Similarly, the shape list form of $C$ is
      \[C=(\BN\cup \ol B,0)\sqcup (\ol F,\wt f),\]
      so that its power series is
      \[C(x)=q^{\BN\cup\ol B}x^0+q^{\ol F}x^{\wt f}.\]
      For short let us write $K\coloneqq m\BN\cup B$, so that
      \[C(x)=q^{\ol K}x^0+q^{\ol F}x^{\wt f}.\]
      
       We can then directly compute each of the $C+V_i$, $C+V_{k+1}$, and $C+V_\BZ$ by computing the power series via multiplication. Indeed,
      \begin{align*}
          (C+V_i)(x)&=C(x)V_i(x)\\
          &=\Bigpr{q^{\ol K}x^0+q^{\ol F}x^{\wt f}}\Bigpr{q^{\ol W_{i,1}}x^{(2i-2)\wt f}+\cdots+q^{\ol W_{i,\wt f}}x^{(2i-1)\wt f-1}+q^{0}x^{(2i-1)\wt f}+\cdots+q^{0}x^{2i\wt f-1}}\\
          &=q^{\ol K}q^{\ol W_{i,1}}x^{(2i-2)\wt f}+\cdots + q^{\ol K}q^{\ol W_{i,\wt f}}x^{(2i-1)\wt f-1}\\
          &\qquad +(q^{\ol K}+q^{\ol F}q^{\ol W_{i,1}})x^{(2i-1)\wt f}+\cdots +(q^{\ol K}+q^{\ol F}q^{\ol W_{i,\wt f}}) x^{2i\wt f-1}\\
          &\qquad\qquad +q^{\ol F}x^{2i\wt f}+\cdots+q^{\ol F}x^{(2i+1)\wt f-1}.
      \end{align*}
      Note that, by shifting the index up by one, we have
      \begin{align*}
           (C+V_{i+1})(x)&=q^{\ol K}q^{\ol W_{i+1,1}}x^{2i\wt f}+\cdots + q^{\ol K}q^{\ol W_{i+1,\wt f}}x^{(2i+1)\wt f-1}\\
          &\qquad +(q^{\ol K}+q^{\ol F}q^{\ol W_{i+1,1}})x^{(2i+1)\wt f}+\cdots +(q^{\ol K}+q^{\ol F}q^{\ol W_{i+1,\wt f}}) x^{(2i+2)\wt f-1}\\
          &\qquad\qquad +q^{\ol F}x^{(2i+2)\wt f}+\cdots+q^{\ol F}x^{(2i+3)\wt f-1},
      \end{align*}
      wherein the first $\wt f$ terms of $(C+V_{i+1})(x)$ will combine with the last $\wt f$ terms of $(C+V_i)(x)$. This is the ``cars fitting together'' we were talking about. 
      
      Similarly, we may compute
      \begin{align*}
          (C+V_{k+1})(x)&=C(x)V_{k+1}(x)\\
          &=\Bigpr{q^{\ol K}x^0+q^{\ol F}x^{\wt f}}\Bigpr{q^{\ol W_{k+1,1}}x^{2k\wt f}+\cdots+q^{\ol W_{k+1,r}}x^{2k\wt f+r-1}+q^{0}x^{(2k+1)\wt f}+\cdots+q^{0}x^{(2k+1)\wt f+r-1}}\\
          &=q^{\ol K}q^{\ol W_{k+1,1}}x^{2k\wt f}+\cdots+ q^{\ol K}q^{\ol W_{k+1,r}} x^{2k\wt f+r-1}\\
          &\qquad +(q^{\ol K}+q^{\ol F}q^{\ol W_{k+1,1}})x^{(2k+1)\wt f}+\cdots+(q^{\ol K}+q^{\ol F}q^{\ol W_{k+1,r}})x^{(2k+1)\wt f+r-1}\\
          &\qquad\qquad +q^{\ol F}x^{(2k+2)\wt f}+\cdots+ q^{\ol F}x^{(2k+2)\wt f+r-1}
      \end{align*}
      and
      \begin{align*}
          (C+V_{\BZ})(x)&=C(x)V_\BZ(x)\\
          &=\Bigpr{q^{\ol K}x^0+q^{\ol F}x^{\wt f}}\Bigpr{q^\BZ x^{2k\wt f+r}+\cdots+q^\BZ x^{(2k+1)\wt f-1}+q^\BZ x^{(2k+2)\wt f}+\cdots+q^\BZ x^{m-1}}\\
          &=q^{\ol K}q^{\BZ}x^{2k\wt f+r}+\cdots+q^{\ol K}q^{\BZ}x^{(2k+1)\wt f-1}\\
          &\qquad +q^{\ol F}q^{\BZ}x^{(2k+1)\wt f+r}+\cdots+q^{\ol F}q^\BZ x^{(2k+2)\wt f-1}\\
          &\qquad\qquad +q^{\ol K}q^\BZ x^{(2k+2)\wt f}+\cdots+ q^{\ol K}q^\BZ x^{m-1}\\
          &\qquad\qquad \qquad +q^{\ol F}q^\BZ x^{(2k+3)\wt f}+\cdots+ q^{\ol F}q^\BZ x^{m-1+\wt f}\\
           &=q^{\ol K}q^{\BZ}x^{2k\wt f+r}+\cdots+q^{\ol K}q^{\BZ}x^{(2k+1)\wt f-1}\\
          &\qquad +q^{\ol F}q^{\BZ}x^{(2k+1)\wt f+r}+\cdots+q^{\ol F}q^\BZ x^{(2k+2)\wt f-1}\\
          &\qquad\qquad +q^{\ol K}q^\BZ x^{(2k+2)\wt f}+\cdots+ q^{\ol K}q^\BZ x^{(2k+3)\wt f-1}\\
          &\qquad\qquad \qquad +(q^{\ol K}q^\BZ+q^{\ol F}q^\BZ) x^{(2k+3)\wt f}+\cdots+ (q^{\ol K}q^\BZ+q^{\ol F}q^\BZ) x^{m-1}\\
          &\qquad\qquad\qquad\qquad+ q^{\ol F}q^\BZ x^m+\cdots+q^{\ol F}q^\BZ x^{m+\wt f-1}\\
          &=q^{\ol K}q^{\BZ}x^{2k\wt f+r}+\cdots+q^{\ol K}q^{\BZ}x^{(2k+1)\wt f-1}\\
          &\qquad +q^{\ol F}q^{\BZ}x^{(2k+1)\wt f+r}+\cdots+q^{\ol F}q^\BZ x^{(2k+2)\wt f-1}\\
          &\qquad\qquad +q^{\ol K}q^\BZ x^{(2k+2)\wt f}+\cdots+ q^{\ol K}q^\BZ x^{(2k+3)\wt f-1}\\
          &\qquad\qquad \qquad +(q^{\ol K}q^\BZ+q^{\ol F}q^\BZ) x^{(2k+3)\wt f}+\cdots+ (q^{\ol K}q^\BZ+q^{\ol F}q^\BZ) x^{m-1}\\
          &\qquad\qquad\qquad\qquad+ q^{\ol F}q^\BZ q^1 x^0+\cdots+q^{\ol F}q^\BZ q^1x^{\wt f-1}.
      \end{align*}
      
      Having calculated the sums of $C$ with each of the blocs, we may now calculate $C+V$. Since $V=V_1\sqcup\cdots \sqcup V_k\sqcup V_{k+1}\sqcup V_\BZ$, we have
      \[(C+V)(x)=(C+V_1)(x)+\cdots+(C+V_k)(x)+(C+V_{k+1})(x)+(C+V_{\BZ})(x),\]
      so that adding up the above results we get
      \begin{align*}
          (C+V)(x)&=q^{\ol K}q^{\ol W_{1,1}}x^0+\cdots+q^{\ol K}q^{\ol W_{1,\wt f}}x^{\wt f-1}\\
          &\quad +\sum_{i=1}^{k-1}\bigg((q^{\ol K}+q^{\ol F}q^{\ol W_{i,1}})x^{(2i-1)\wt f}+\cdots +(q^{\ol K}+q^{\ol F}q^{\ol W_{i,\wt f}}) x^{2i\wt f-1}\\
          &\qquad\qquad\quad +(q^{\ol F}+q^{\ol K}q^{\ol W_{i+1,1}})x^{2i\wt f}+\cdots+(q^{\ol F}+q^{\ol K}q^{\ol W_{i+1,\wt f}})x^{(2i+1)\wt f-1}\bigg)\\
          &\quad +(q^{\ol K}+q^{\ol F}q^{\ol W_{k,1}})x^{(2k-1)\wt f}+\cdots +(q^{\ol K}+q^{\ol F}q^{\ol W_{k,\wt f}}) x^{2k\wt f-1}\\
          &\qquad\quad +(q^{\ol F}+q^{\ol K}q^{\ol W_{k+1,1}})x^{2k\wt f}+\cdots+(q^{\ol F}+q^{\ol K}q^{\ol W_{k+1,r}})x^{2k\wt f+r-1}\\
          &\qquad\quad+(q^{\ol F}+q^{\ol K}q^{\BZ} )x^{2k\wt f+r}+\cdots+(q^{\ol F}+q^{\ol K}q^{\BZ} )x^{(2k+1)\wt f-1}\\
          &\quad +(q^{\ol K}+q^{\ol F}q^{\ol W_{k+1,1}})x^{(2k+1)\wt f}+\cdots+(q^{\ol K}+q^{\ol F}q^{\ol W_{k+1,r}})x^{(2k+1)\wt f+r-1}\\
          &\qquad\quad +q^{\ol F}q^{\BZ}x^{(2k+1)\wt f+r}+\cdots+q^{\ol F}q^\BZ x^{(2k+2)\wt f-1}\\
          &\qquad\quad +q^{\ol K}q^\BZ x^{(2k+2)\wt f}+\cdots+ q^{\ol K}q^\BZ x^{(2k+3)\wt f-1}\\
          &\qquad\quad +(q^{\ol K}q^\BZ+q^{\ol F}q^\BZ) x^{(2k+3)\wt f}+\cdots+ (q^{\ol K}q^\BZ+q^{\ol F}q^\BZ) x^{m-1}\\
          &\qquad\quad+ q^{\ol F}q^\BZ q^1 x^0+\cdots+q^{\ol F}q^\BZ q^1x^{\wt f-1}\\
          &=q^{(\ol K+\ol W_{1,1})\sqcup(\ol F+\BZ+1)}x^0+\cdots+q^{(\ol K+\ol W_{1,\wt f})\sqcup(\ol F+\BZ+1)}x^{\wt f-1}\\
          &\quad +\sum_{i=1}^{k-1}\bigg(q^{\ol K\sqcup (\ol F+\ol W_{i,1})}x^{(2i-1)\wt f}+\cdots +q^{\ol K\sqcup (\ol F+\ol W_{i,\wt f})} x^{2i\wt f-1}\\
          &\qquad\qquad\quad +q^{\ol F\sqcup (\ol K+\ol W_{i+1,1})}x^{2i\wt f}+\cdots+q^{\ol F\sqcup (\ol K+\ol W_{i+1,\wt f})}x^{(2i+1)\wt f-1}\bigg)\\
          &\quad +q^{\ol K\sqcup (\ol F+\ol W_{k,1})}x^{(2k-1)\wt f}+\cdots +q^{\ol K\sqcup (\ol F+\ol W_{k,\wt f})} x^{2k\wt f-1}\\
          &\qquad\quad +q^{\ol F\sqcup (\ol K+\ol W_{k+1,1})}x^{2k\wt f}+\cdots+q^{\ol F\sqcup (\ol K+\ol W_{k+1,r})}x^{2k\wt f+r-1}\\
          &\qquad\quad+q^{\ol F\sqcup (\ol K+\BZ) }x^{2k\wt f+r}+\cdots+q^{\ol F\sqcup (\ol K+\BZ) }x^{(2k+1)\wt f-1}\\
          &\quad +q^{\ol K\sqcup(\ol F+\ol W_{k+1,1})}x^{(2k+1)\wt f}+\cdots+q^{\ol K\sqcup(\ol F+\ol W_{k+1,r})}x^{(2k+1)\wt f+r-1}\\
          &\qquad\quad +q^{\ol F+\BZ}x^{(2k+1)\wt f+r}+\cdots+q^{\ol F+\BZ} x^{(2k+2)\wt f-1}\\
          &\qquad\quad +q^{\ol K+\BZ} x^{(2k+2)\wt f}+\cdots+ q^{\ol K+\BZ} x^{(2k+3)\wt f-1}\\
          &\qquad\quad +q^{(\ol K+\BZ)\sqcup(\ol F+\BZ)} x^{(2k+3)\wt f}+\cdots+ q^{(\ol K+\BZ)\sqcup(\ol F+\BZ)} x^{m-1}
      \end{align*}
      
      It should be well noted that the purpose of the bound
      \[m\ge (2k+1)\wt f+r\]
      is to ensure that this expression has all like terms combined; in particular, it ensures that all terms of form
      \[q^{\ol K\sqcup(\ol F+\ol W_{i,j})}x^{(2i-1)\wt f+j-1}\]
      have exponents lying in the range $[0,m-1]$. 
      
      \medskip
      
      \textit{(4) Deciphering}. Having computed $(C+V)(x)$, to check that $C$ is an additive complement to $V$, i.e. $C+V=\BZ$, it suffices to check that every coefficient $[x^i](C+V)(x)$ is of form $q^A$ where $A\supseteq\BZ$, i.e. $q^\BZ\le [x^i](C+V)(x)$ for all $i$. We can verify this by considering casework depending on the range of the exponent:
      \begin{enumerate}[(i)]
          \item $x^0$ to $x^{\wt f-1}$. Coefficients in this range are of form $q^{(\ol K+\ol W_i)\sqcup(\ol F+\BZ+1)}$;
          since $\BZ\subseteq\ol F+\BZ+1$, it is clear that 
          \[q^{\BZ}\le q^{(\ol K+\ol W_i)\sqcup(\ol F+\BZ+1)}.\]
          \item $x^{(2i-1)\wt f}$ to $x^{(2i+1)\wt f-1}$ for $1\le i\le k-1$. Coefficients in this range are either of form 
         $q^{\ol K\sqcup(\ol F+\ol W_i)}$ or $q^{\ol F\sqcup (\ol K+\ol W_i)}$.
          In the latter case, since $\BN\subseteq\ol K$ and since $\ol W_i$ contains infinitely many negative elements, we have $\BZ\subseteq \ol K+\ol W_i$ and therefore 
          \[q^{\BZ}\le q^{\ol F\sqcup (\ol K+\ol W_i)}.\]
          In the former case, recall $\ol F+\ol W_i=\BZ\sm\ol S_i$, and since $\ol S_i\subseteq\ol K$, we have $\BZ\subseteq \ol K\sqcup(\ol F+\ol W_i)$, i.e. 
          \[q^\BZ\le q^{\ol K\sqcup(\ol F+\ol W_i)}\]
          as well.
          \item $x^{(2k-1)\wt f}$ to $x^{2k\wt f+r-1}$. Coefficients in this range are of the same form as (ii), so we are done here for the same reasons as in (ii).
          \item $x^{2k\wt f+r}$ to $x^{(2k+1)\wt f-1}$. Coefficients in this range are of form $q^{\ol F\sqcup(\ol K+\BZ)}$. Since clearly $\BZ\subseteq\ol K+\BZ$, we have
          \[q^\BZ\le q^{\ol F\sqcup(\ol K+\BZ)}.\]
          \item $x^{(2k+1)\wt f}$ to $x^{(2k+1)\wt f+r-1}$. Coefficients in this range are covered by (ii) also.
          \item $x^{(2k+1)\wt f+r}$ to $x^{m-1}$. Coefficients in this range are either of form $q^{\ol F+\BZ}$, $q^{\ol K+\BZ}$, or $q^{\ol F+\BZ}+q^{\ol K+\BZ}$. It is clear that $\BZ\subseteq \ol F+\BZ,\ol K+\BZ$, so that
          \[q^\BZ\le q^{\ol F+\BZ},q^{\ol K+\BZ},q^{\ol F+\BZ}+q^{\ol K+\BZ}.\]
      \end{enumerate}
      This concludes the check that $C+V=\BZ$.
      
      Next let us see that all elements of $C$ have dependent elements in $\BZ$ with respect to $V$. In fact, we claim that every column labeled by $(2i-1)\wt f+j-1$ (where $1\le j\le \wt f$ and $(i-1)\wt f+j\le n$) contains dependent elements of every element of $S_{i,j}$ and every element of $F$. Indeed, the coefficients $[x^{(2i-1)\wt f+j-1}](C+V)(x)$ are of form
      \[[x^{(2i-1)\wt f+j-1}](C+V)(x)=q^{\ol K\sqcup(\ol F+\ol W_{i,j})};\]
      by construction we have $\ol F+\ol W_{i,j}=\BZ\sm \ol S_{i,j}$ in such a way that there are dependent elements of every element of $\ol F$ of form $\ol\delta<\ol y_{i,j}\le \min \ol B$, so that in particular $\ol\delta\not\in\ol K$. Then, in $\ol K\sqcup (\ol F+\ol W_{i,j})$, each such $\ol\delta$ is counted/covered exactly once, so that they are still dependent elements of $\ol F$ in the equation 
      \[\ol K\sqcup (\ol F+\ol W_{i,j})\supseteq \BZ,\]
      which gives dependents of all of $F$. Similarly, $\ol S_{i,j}$ is avoided by $\ol F+\ol W_{i,j}$ (which is equal to $\BZ\sm \ol S_{i,j}$ actually, namely everyone but $S_{i,j}$), and is covered exactly once in $K$, so that we have the dependent elements of $S_{i,j}$ also. Taking the union over all appropriate $i,j$, this gives the dependent elements of all $\bigcup_{i,j}S_{i,j}=m\BN\cup B$, so that all elements of $C$ have dependents in the equation 
      \[C+V=\BZ.\]
      This concludes the check that $C$ is minimal with respect to $C+V=\BZ$.
      
      Lastly we should remark on the use of the bound 
      \[m\ge (2k+1)\wt f+r.\]
      Indeed, if this were to fail, i.e. if $m<(2k+1)\wt f+r$, then the term in $(C+V)(x)$
      \[q^{\ol K\sqcup(\ol F+\ol W_{k+1,r})}x^{(2k+1)\wt f+r-1}\]
      would be reduced in the quotient which defines $\Xi_m$ to $q^{\bigpr{\ol K\sqcup(\ol F+\ol W_{k+1,r})}+1}x^{(2k+1)\wt f+r-1-m}$, which would then ``collide'' (i.e. combine like terms) with a term from earlier; in this case, we can no longer guarantee that $\ol S_{i,j}\subseteq\BZ$ is covered only once in the coefficient $[x^{(2k+1)\wt f+r-1-m}](C+V)(x)$. 
      %by a number in
    %   \[\bigcup_{i=1}^k [(2i-1)\wt f,2i\wt f-1]\cup [(2k+1)\wt f,(2k+1)\wt f+r-1]\]
\end{proof}

\subsection{Proving Theorem 7}
We next prove Theorem 7. The idea is similar to Theorem 4, where we put ``passengers'' (i.e. the columns with dependents of $C$) into ``cars'', which we put in a large ``garage'' (i.e. $Z_m$), and somehow if the garage is long enough then all the cars will fit and none of the passengers will overlap. 

In Theorem 4, because the problem conditions are specific (e.g. only one column of $F$ and $B$), each car can fit many passengers inside. However, because the distribution of points in $C$ is not known in the setting of Theorem 7, we will only be able to put one passenger in each car this time.

% The intuitive picture here is something like:
% \begin{center}
%     \textcolor{red}{include graphic here}
% \end{center}
\begin{proof}[Proof of Theorem 7]
     This theorem is stated in the setting of $\BZ^2$ projecting to $Z_m$. However, in proving the theorem we will work directly in $Z_m$, taking the assumptions that there is some set cover $\cal S_i$ of the infinite columns $K_i\subseteq C\subseteq Z_m$ 
     \[S_{i,1}\cup\cdots\cup S_{i,n_i}=K_i\]
     such that, for each $i,j$, there are (possibly empty) $W_{i,j;\mu},U_{i,j;\nu}\subseteq Z_m$ such that
     \[\Col(K_i)\sm S_{i,j}=\bigcup_{\mu=1}^t (F_\mu +W_{i,j;\mu})\cup\bigcup_{\nu=1}^r(K_\nu+U_{i,j;\nu}),\]
     and that
     \[m\ge(\ell+1)\pr*{|F_/|+\sum_{i=1}^r |\cal S_i|},\]
     i.e.
     \[m\ge(\ell+1)\pr*{t+\sum_{i=1}^r n_i}\vrot{\coloneqq}{180} (\ell+1)N.\]
     
     Each $F_i$ or $K_j$ will reside in the column labeled by $\wt f_i$ or $\wt k_j$ respectively, and we may sometimes write $\wt c_i$ for the label of the column containing $C_i$, the $i$-th column of $C$ counting from the left for $1\le i\le t+r$. 

     As before, the proof will be separated into italicized sections for ease of reading.
     
     \medskip  
     
     \textit{(1) Observations}. As per the definition of $\ell$, let\footnote{This notation is provocatively chosen.} $Q_/$ be a set such that
     \[C_/+Q_/=[0,\ell-1].\]
     The discussion immediately following the statement of Theorem 7 shows that such a set exists with the given upper bound $\ell\le\te{outerrange}(\wh C_/)+\te{innerrange}(\wh C_/)$. Then, define
     \[Q(x)\coloneqq \sum_{\wt q\in Q_/}q^\BZ x^{\wt q}.\]
     Note well that the set $Q$ corresponding to this power series satisfies
     \[(C+Q)(x)= q^{P_0} x^0+\cdots+q^{P_{\ell-1}} x^{\ell-1}\ge q^{\BZ} x^0+\cdots+q^{\BZ} x^{\ell-1},\]
     where each $P_i\ge\BZ$. 
     
     Since the columns $F_1,\cdotsc,F_t$ are finite sets, by the theorem of Kwon we know that there exists\footnote{We apologize in advance for the overload of notation for the letter $W$ here, but remember that $W_i$ with a single subscript is the additive complement of a $F_i$, while $W_{i,j;\mu}$ with three subscripts is the set involved in the theorem assumptions.} $W_i\subseteq\BZ$ such that
     \[F_i+W_i=\BZ\]
     and $F_i$ is minimal with respect to this equation. 
     
     Consider the ordered list of symbols (which each stand for one of our sets)
     \[\big\{F_1,\cdotsc,F_t,S_{1,1},\cdotsc,S_{1,n_1},\cdotsc,S_{r,1},\cdotsc,S_{r,n_r}\big\};\]
     for a symbol $S$ in this list, let $\alpha(S)$ be the index in this list where $S$ appears. For example, $\alpha(F_1)=1$, $\alpha(F_t)=t$, and $\alpha(S_{r,n_r})=t+\sum_i n_i$. We will use the shorthand
     \[\alpha_{i,j}\coloneqq \alpha(i,j)\coloneqq \alpha(S_{i,j})=\te{the index where $S_{i,j}$ appears}.\]
     In line with this notation, we will define\footnote{We apologize in advance for overloading notation, but the way to distinguish these $S_i$ and the members of the set cover $S_{i,j}\in \cal S_i$ is that the former has only one subscript while the latter has two.} $S_i$ for $1\le i\le t+\sum_j n_j=N$ to be
     \[\big\{S_1,\cdotsc,S_N\big\}\coloneqq\big\{F_1,\cdotsc,F_t,S_{1,1},\cdotsc,S_{1,n_1},\cdotsc,S_{r,1},\cdotsc,S_{r,n_r}\big\},\]
     so that $S_{\alpha(i,j)}=S_{i,j}$.
     
     \medskip
     
     \textit{(2) Construction}. First we give the construction of the set $V$ to which $C$ shall be a MAC. Rather than give the explicit set construction, we will give the power series $V(x)$, which as remarked earlier determines the set $V$. Since $V$ shall be quite unwieldy, we will again break it up into ``blocs''. 
     
    %  Roughly, a single ``magazine'' shifted to begin at 0 will look either like
    %  \[q^{\ol W_i}x^{\ell-\wt f_i}+Q(x)\]
    %  if the ``bullet'' inside the magazine are dependents of $F_i$, or
    %  \[q^0x^{\ell-\wt k_i}+\sum_{\mu=1}^t q^{\ol W_{i,j;\mu}}x^{\ell-\wt f_\mu}+\sum_{\nu=1}^r q^{\ol U_{i,j;\nu}}x^{\ell-\wt k_\nu}+Q(x)\]
    %  if the bullet are dependents of $S_{i,j}$. 
     
    %  The ``magazine bloc'' will be built by putting all these single magazines together. 
     
    %  More precisely, the ``magazine blocs'' is defined by the generating function
    
    The ``car blocs'' $V_i$ ($1\le i\le t+\sum_j n_j=N$) are defined by formal power series which we give here. For $1\le i\le t$, define
     \[V_i(x)\coloneqq x^{(i-1)(\ell+1)}\Bigpr{q^{\ol W_i}x^{\ell-\wt f_i}+Q(x)},\]
    and for $t+1\le \alpha(i,j)\le N$ let us define      
    \[V_{\alpha(i,j)}(x)\coloneqq x^{(\alpha(i,j)-1)(\ell+1)}\pr*{q^0 x^{\ell-\wt k_i}+\sum_{\mu=1}^t q^{\ol W_{i,j;\mu}}x^{\ell-\wt f_\mu}+\sum_{\nu=1}^r q^{\ol U_{i,j;\nu}}x^{\ell-\wt k_\nu}+Q(x)}.\]
    Note well that some of these terms in $\sum_{\mu=1}^t q^{\ol W_{i,j;\mu}}x^{\ell-\wt f_\mu}+\sum_{\nu=1}^r q^{\ol U_{i,j;\nu}}x^{\ell-\wt k_\nu}$ could be zero, for example if $\ol W_{i,j;\mu}=\emptyset$ then $q^{\ol W_{i,j;\mu}}=0$. 
     
    Similarly to last time, the ``filler blocs'' will be defined by the power series 
     \[V_\BZ(x)\coloneqq (x^{N(\ell+1)}+\cdots+x^{m-1})Q(x).\]
     In the case that $m=N(\ell+1)$, this power series is defined to be zero. 
     
     Then we shall define $V$ to be
     \[V\coloneqq V_1\sqcup \cdots\sqcup V_N\sqcup V_\BZ,\]
     whose power series shall thus be the sum of those written above.
     
     The claim then is that $C+V=\BZ$, and that $C$ is minimal with respect to this condition.
     
     \medskip 
     
     \textit{(3) Generatingfunctionology}. We will compute the power series $(C+V)(x)$ by computing $(C+V_i)(x)$ and $(C+V_\BZ)(x)$ and then adding them together.
     
     Note that in this setup, the power series of $C=F_1\sqcup\cdots\sqcup F_t\sqcup K_1\sqcup\cdots\sqcup K_r$ is given by
     \[C(x)=q^{\ol F_1}x^{\wt f_1}+\cdots+ q^{\ol F_t}x^{\wt f_t}+ q^{\ol K_1}x^{\wt k_1}+\cdots+ q^{\ol K_r}x^{\wt k_r}.\]
     
     Let us then compute the sums $C+V_i$ and $C+V_\BZ$ by computing the formal power series thereof. For $1\le i\le t$, compute\nc{\Err}{\on{Err}}\nc{\Dep}{\on{Dep}}\nc{\Fil}{\on{Fil}}
     \begin{align*}
         (C+V_i)(x)&=C(x)V_i(x)\\
         &=\bigpr{q^{\ol F_i}x^{\wt f_i}+(C\sm F_i)(x)}x^{(i-1)(\ell+1)}\Bigpr{q^{\ol W_i}x^{\ell-\wt f_i}+Q(x)}\\
         &=q^{\ol F_i}q^{\ol W_i}x^{i\ell+i-1}+(C\sm F_i)(x)q^{\ol W_i}x^{i\ell+i-1-\wt f_i}+x^{(i-1)(\ell+1)}(C+Q)(x)\\
         &=q^{P_0}x^{(i-1)\ell+i-1}+\cdots+q^{P_{\ell-1}}x^{i\ell +i-2}\\
         &\qquad +q^{\ol F_i}q^{\ol W_i}x^{i\ell+i-1}\\
         &\qquad\qquad+(C\sm F_i)(x)q^{\ol W_i}x^{i\ell+i-1-\wt f_i}\\
         &\vrot{\coloneqq}{180}\Fil_i(x) \\
         &\qquad +\Dep_i(x)\\
         &\qquad\qquad+\Err_i(x),
     \end{align*}
     where we have defined (Fil standing for Filler, Dep standing for Dependents, and Err standing for Error)
     \begin{align*}
         \Fil_i(x)&\coloneqq q^{P_0}x^{(i-1)\ell+i-1}+\cdots+q^{P_{\ell-1}}x^{i\ell +i-2},\\
         \Dep_i(x)&\coloneqq q^{\ol F_i}q^{\ol W_i}x^{i\ell+i-1},\\
         \Err_i(x)&\coloneqq (C\sm F_i)(x)q^{\ol W_i}x^{i\ell+i-1-\wt f_i}.
     \end{align*}
     Note well that the exponents of $x$ appearing in $\Fil_i(x)$ are precisely
     \[\Rng\Fil_i(x)=\big[(i-1)\ell+i-1,i\ell+i-2\big],\]
     while the exponent appearing in $\Dep_i(x)$ is
     \[\Rng\Dep_i(x)=\{i\ell+i-1\}.\]
     
     Recall that $P_i\supseteq\BZ$, and note well that all terms in the last summand $\Err_i(x)$ are constant multiples of
     \[x^{i\ell+i-1-\wt f_i+\wt f_j} \quad\te{or}\quad x^{i\ell+i-1-\wt f_i+\wt k_j}\]
     for $j\neq i$, so that in particular
     \[[x^{i\ell+i-1}]{\Err_i(x)}=0.\]
     Also note well that the largest power of $x$ appearing in $\Err_i(x)$ is $x^{i\ell+i-1-\wt f_i+\max C_/}$, which satisfies
     \[{i\ell+i-1-\wt f_i+\max C_/}\le (i+1)\ell+i-1\]
     since $\ell> \max C_/-\min C_/$, and the smallest power of $x$ appearing in $\Err_i(x)$ is $x^{i\ell+i-1-\wt f_i+\min C_/}$, which satisfies
     \[i\ell+i-1-\wt f_i+\min C_/\ge(i-1)\ell+i-1\]
     since $\wt f_i\le \max C_/$ and $\ell>\max C_/-\min C_/$. These two inequalities ensure that when we add $(C+V_i)(x)+(C+V_{i+1})(x)$, the summand $\Err_i(x)$ from $C+V_i$ will combine with the terms $\Fil_i(x)$ and $\Fil_{i+1}(x)$ and will not collide\footnote{I.e., will not ``collide with the passenger from $C+V_{i+1}$''.} with the term $\Dep_{i+1}(x)=q^{\ol F_{i+1}}q^{\ol W_{i+1}}x^{(i+1)\ell+i}$ from $C+V_{i+1}$.

     Similarly, for $t+1\le \alpha(i,j)\le N$, 
     \begin{align*}
         (C+V_{\alpha(i,j)})(x)&=C(x)V_{\alpha(i,j)}(x)\\
         &=C(x)x^{(\alpha_{i,j}-1)(\ell+1)}\pr*{q^0 x^{\ell-\wt k_i}+\sum_{\mu=1}^t q^{\ol W_{i,j;\mu}}x^{\ell-\wt f_\mu}+\sum_{\nu=1}^r q^{\ol U_{i,j;\nu}}x^{\ell-\wt k_\nu}+Q(x)}\\
         &=\bigpr{q^{\ol K_i}x^{\wt k_i}+(C\sm K_i)(x)}x^{(\alpha_{i,j}-1)(\ell+1)}x^{\ell-\wt k_i}\\
         &\qquad +\sum_{\mu=1}^t \bigpr{q^{\ol F_\mu}x^{\wt f_\mu}+(C\sm F_\mu)(x)}x^{(\alpha_{i,j}-1)(\ell+1)} q^{\ol W_{i,j;\mu}}x^{\ell-\wt f_\mu}\\
         &\qquad +\sum_{\nu=1}^r \bigpr{q^{\ol K_\nu}x^{\wt k_\nu}+(C\sm K_\nu)(x)}x^{(\alpha_{i,j}-1)(\ell+1)} q^{\ol U_{i,j;\nu}}x^{\ell-\wt k_\nu}\\
         &\qquad\qquad+ x^{(\alpha_{i,j}-1)(\ell+1)}(C+Q)(x)\\
         &=q^{P_0}x^{(\alpha_{i,j}-1)\ell+\alpha_{i,j}-1}+\cdots+q^{P_{\ell-1}}x^{\alpha_{i,j}\ell+\alpha_{i,j}-2}\\
         &\qquad +\pr*{q^{\ol K_i}+\sum_{\mu=1}^t q^{\ol F_\mu}q^{\ol W_{i,j;\mu}}+\sum_{\nu=1}^r q^{\ol K_\nu}q^{\ol U_{i,j;\nu}}}x^{\alpha_{i,j}\ell+\alpha_{i,j}-1}\\
         &\qquad\qquad+ (C\sm K_i)(x)x^{\alpha_{i,j}\ell+\alpha_{i,j}-1-\wt k_i}\\
         &\qquad\qquad+\sum_{\mu=1}^t (C\sm F_\mu)(x)q^{\ol W_{i,j;\mu}}x^{\alpha_{i,j}\ell+\alpha_{i,j}-1-\wt f_\mu}\\
         &\qquad\qquad +\sum_{\nu=1}^r (C\sm K_\nu)(x)q^{\ol U_{i,j;\nu}}x^{\alpha_{i,j}\ell+\alpha_{i,j}-1-\wt k_\nu}\\
         &\vrot{\coloneqq}{180} \Fil_{\alpha(i,j)}(x)\\
         &\qquad +\Dep_{\alpha(i,j)}(x)\\
         &\qquad\qquad+ \Err_{\alpha(i,j)}(x),
     \end{align*}
     where we have defined
     \begin{align*}
         \Fil_{\alpha(i,j)}(x)&\coloneqq  q^{P_0}x^{(\alpha_{i,j}-1)\ell+\alpha_{i,j}-1}+\cdots+q^{P_{\ell-1}}x^{\alpha_{i,j}\ell+\alpha_{i,j}-2},\\
         \Dep_{\alpha(i,j)}(x)&\coloneqq \pr*{q^{\ol K_i}+\sum_{\mu=1}^t q^{\ol F_\mu}q^{\ol W_{i,j;\mu}}+\sum_{\nu=1}^r q^{\ol K_\nu}q^{\ol U_{i,j;\nu}}}x^{\alpha_{i,j}\ell+\alpha_{i,j}-1},\\
         \Err_{\alpha(i,j)}(x)&\coloneqq  (C\sm K_i)(x)x^{\alpha_{i,j}\ell+\alpha_{i,j}-1-\wt k_i}+\sum_{\mu=1}^t (C\sm F_\mu)(x)q^{\ol W_{i,j;\mu}}x^{\alpha_{i,j}\ell+\alpha_{i,j}-1-\wt f_\mu} \\
         &\qquad\qquad\qquad\hspace{10em}+\sum_{\nu=1}^r (C\sm K_\nu)(x)q^{\ol U_{i,j;\nu}}x^{\alpha_{i,j}\ell+\alpha_{i,j}-1-\wt k_\nu}.
     \end{align*}
     Again note well that the exponents appearing in $\Fil_{\alpha(i,j)}(x)$ are 
     \[\Rng\Fil_{\alpha(i,j)}(x)=\big[ (\alpha_{i,j}-1)\ell+\alpha_{i,j}-1,\alpha_{i,j}\ell+\alpha_{i,j}-2\big],\]
     while the exponent appearing in $\Dep_{\alpha(i,j)}(x)$ is
     \[\Rng\Dep_{\alpha(i,j)}(x)=\{\alpha_{i,j}\ell+\alpha_{i,j}-1\}.\]
     
     As before we remark that the coefficients in $\Fil_{\alpha(i,j)}(x)$ have $P_\mu\supseteq \BZ$, and we must note well that all the terms in the last summand $\Err_{\alpha(i,j)}(x)$ are multiples of 
     \[x^{\alpha_{i,j}\ell+\alpha_{i,j}-1-\wt c_\mu+\wt c_\nu},\]
     where for example $\wt c_\mu=\wt k_i$ for the first summand  in $\Err_{\alpha(i,j)}(x)$ and $\wt c_\nu\neq \wt c_\mu$. In particular this means
     \[[x^{\alpha_{i,j}\ell+\alpha_{i,j}-1}]\Err_{\alpha(i,j)}(x)=0.\]
     Also note well that the largest power of $x$ appearing in $\Err_{\alpha(i,j)}(x)$ is at most (in the sense of comparing the exponents) $x^{\alpha_{i,j}\ell+\alpha_{i,j}-1-\min C_/+\max C_/}$, which satisfies
     \[\alpha_{i,j}\ell+\alpha_{i,j}-1-\min C_/+\max C_/\le (\alpha_{i,j}+1)\ell+\alpha_{i,j}-1\]
     since $\ell>\max C_/-\min C_/$. Moreover the smallest power of $x$ appearing in $\Err_{\alpha(i,j)}(x)$ is at least
     \[\alpha_{i,j}\ell+\alpha_{i,j}-1-\max C_/+\min C_/\ge (\alpha_{i,j}-1)\ell+\alpha_{i,j}-1\]
     for the same reason. These two inequalities ensure that when we consider the sum $(C+V_{\alpha(i,j)})(x)+(C+V_{\alpha(i,j)+1})(x)$, the summand $\Err_{\alpha(i,j)}(x)$ from $(C+V_{\alpha(i,j)})(x)$ will combine with the terms $\Fil_{\alpha(i,j)}(x)$ and $\Fil_{\alpha(i,j)+1}(x)$ and will not collide (i.e., after using the relation $x^m=q^1$ so that all exponents are in the range $[0,m-1]$, the set of exponents appearing in the former is disjoint from the set of exponents appearing in the latter) with the term $\Dep_{\alpha(i,j)+1}(x)$.
     
     Lastly let us compute
     \begin{align*}
         (C+V_\BZ)(x)&=C(x)V_\BZ(x)\\
         &=C(x)Q(x)(x^{N(\ell+1)}+\cdots+x^{m-1})\\
         &=\bigpr{q^{P_0}x^0+\cdots+q^{P_{\ell-1}}x^{\ell-1}}\bigpr{x^{N(\ell+1)}+\cdots+x^{m-1}}\\
         &=q^{P_{N(\ell+1)}}x^{N\ell+N}+\cdots+q^{P_{m+\ell-2}}x^{m+\ell-2},
     \end{align*}
     where we the sets $P_{N(\ell+1)},\cdotsc,P_{m+\ell-2}$ are defined such that the last equality holds; i.e., these new\footnote{Note that the subscripts here do not overlap with $P_0,\cdotsc,P_{\ell-1}$, so there's no overlap of notation here.} $P_{N(\ell+1)},\cdotsc,P_{m+\ell-2}$ are obtained by distributing\footnote{Also called ``FOILing''.}. Also note well that, since $P_i\supseteq\BZ$ for $i\in [0,\ell-1]$, we also know $P_i\supseteq \BZ$ for $i\in[N(\ell+1),m+\ell-2]$.
     
     Then, adding everything together, we have
     \begin{align*}
         (C+V)(x)&=(C+V_1)(x)+\cdots+(C+V_N)(x)+(C+V_\BZ)(x)\\
         &=\Fil_1(x)+\Dep_1(x)\\
         &\quad +\sum_{i=2}^{N}\pr*{\bigpr{\Err_{i-1}(x)+\Fil_{i}(x)}+\Dep_{i}(x)}\\
         &\quad +\Err_N(x)+q^{P_{N(\ell+1)}}x^{N\ell+N}+\cdots+q^{P_{m+\ell-2}}x^{m+\ell-2}\\
         &=\big(q^{P_{m}+1}x^0+\cdots+ q^{P_{m+\ell-2}+1}x^{\ell-2}+\Fil_1(x)\big)+\Dep_1(x)\\
         &\quad +\sum_{i=2}^{N}\pr*{\bigpr{\Err_{i-1}(x)+\Fil_{i}(x)}+\Dep_{i}(x)}\\
         &\quad +\Err_N(x)+q^{P_{N(\ell+1)}}x^{N\ell+N}+\cdots+q^{P_{m-1}}x^{m-1}.
     \end{align*}
     Let us decipher what this means now.
     
     \medskip
     
     \textit{(4) Deciphering}. 
     Note that, for $1\le \alpha\le N$, 
     \begin{align*}
         \Rng\Fil_\alpha(x)&=\big[(\alpha-1)\ell+\alpha-1,\alpha\ell+\alpha-2\big],\\
         \Fil_\alpha(x)&\ge q^\BZ x^{(\alpha-1)\ell+\alpha-1}+\cdots+q^{\BZ} x^{\alpha\ell+\alpha-2},\\
         \Rng\Err_\alpha(x)&\subseteq \Rng\Fil_\alpha(x)\cup \Rng\Fil_{\alpha+1}(x),\\
        %  \Rng\big(\Err_{\alpha-1}(x)+\Fil_\alpha(x)\big)&=\big[(\alpha-1)\ell+\alpha-1,\alpha\ell+\alpha-2\big],\\
         \Rng\Dep_\alpha(x)&=\alpha\ell+\alpha-1.
     \end{align*}
     
     If 
     \[m\ge N\ell+N,\]
     then $\Dep_N(x)$ (and therefore all $\Dep_\alpha(x)$ for $\alpha<N$) will not collide, in the quotient defining $\Xi_m$, with the earlier terms in the sum 
     \begin{align*}
         (C+V)(x)&=\big(q^{P_{m}+1}x^0+\cdots+ q^{P_{m+\ell-2}+1}x^{\ell-2}+\Fil_1(x)\big)+\Dep_1(x)\\
         &\quad +\sum_{i=2}^{N}\pr*{\bigpr{\Err_{i-1}(x)+\Fil_{i}(x)}+\Dep_{i}(x)}\\
         &\quad +\Err_N(x)+q^{P_{N(\ell+1)}}x^{N\ell+N}+\cdots+q^{P_{m-1}}x^{m-1}.
     \end{align*}
     That is, the inequality $m\ge N\ell+N$ guarantees that
     \[\Rng\Dep_\alpha(x)\cap \Rng\pr*{(C+V)(x)-\sum_\beta \Dep_\beta(x)}=\emptyset,\]
     i.e. that the only place $x^{\alpha\ell+\alpha-1}$ appears with nonzero coefficient is in $\Dep_\alpha(x)$. 
     
     Let us first check that $C+V=\BZ$. Because of the presence of the $\Fil_\alpha(x)$, we have %that %all coefficients in front of $x^j$ for $j\in \big[(\alpha-1)\ell+\alpha-1,\alpha\ell+\alpha-2\big]$ are of form $q^S$, where $S
     \[[x^j](C+V)(x)\ge q^\BZ\qquad\foralls j\in \big[(\alpha-1)\ell+\alpha-1,\alpha\ell+\alpha-2\big].\]
     
     Note that the intervals $\big[(\alpha-1)\ell+\alpha-1,\alpha\ell+\alpha-2\big]$ and $\big[\alpha\ell+\alpha,(\alpha+1)\ell+\alpha-1\big]$ are separated by $\alpha\ell+\alpha-1$, and this coefficient we can readily see is
     \[[x^{\alpha\ell+\alpha-1}](C+V)(x)=\begin{cases} q^{\ol F_i}q^{\ol W_i}=q^{\ol F_i+\ol W_i}\ge q^\BZ & \hspace*{-8em}\te{if }\alpha=i\le t\\\beau{q^{\ol K_i}+\sum_{\mu=1}^t q^{\ol F_\mu}q^{\ol W_{i,j;\mu}}+\sum_{\nu=1}^r q^{\ol K_\nu}q^{\ol U_{i,j;\nu}}}=q^{\ol K_i\sqcup \bigsqcup_\mu (\ol F_\mu+\ol W_{i,j;\mu})\sqcup\bigsqcup_\nu (\ol K_\nu+\ol U_{i,j;\nu})}\ge q^\BZ & \te{else}\end{cases},\]
    % \begin{align*}
    %     & [x^{\alpha\ell+\alpha-1}](C+V)(x)\\
    %     &\hspace{5em}=\begin{cases} q^{\ol F_i}q^{\ol W_i}=q^{\ol F_i+\ol W_i}\ge q^\BZ & \hspace*{-8em}\te{if }\alpha=i\le t\\\beau{q^{\ol K_i}+\sum_{\mu=1}^t q^{\ol F_\mu}q^{\ol W_{i,j;\mu}}+\sum_{\nu=1}^r q^{\ol K_\nu}q^{\ol U_{i,j;\nu}}}=q^{\ol K_i\sqcup \bigsqcup_\mu (\ol F_\mu+\ol W_{i,j;\mu})\sqcup\bigsqcup_\nu (\ol K_\nu+\ol U_{i,j;\nu})}\ge q^\BZ & \te{else}\end{cases},
    % \end{align*}
     where in the second case\footnote{The coefficient in the second case is at least $q^\BZ$ since $\bigsqcup_\mu (\ol F_\mu+\ol W_{i,j;\mu})\sqcup\bigsqcup_\nu (\ol K_\nu+\ol U_{i,j;\nu})$ covers everything except $\ol S_{i,j}$, and $\ol K_i$ covers $\ol S_{i,j}$.} we have $\alpha=\alpha(i,j)>t$. Hence the coefficient in front of $x$ raised to any power in the range $[0,N\ell+N-1]$ will be greater than or equal to $q^\BZ$. 
     
     Lastly, note that for $N\ell+N\le\mu\le m-1$ we have
     \[[x^\mu](C+V)(x)=q^{P_{\mu}}\ge q^\BZ.\]
     
     This concludes the check that
     \[(C+V)(x)\ge q^\BZ x^0+\cdots+ q^\BZ x^{m-1},\]
     i.e. that $C+V=\BZ$.
     
     Next let us check that $C$ is minimal with respect to $C+V=\BZ$. To do this we will see that all elements of $C$ have dependents in $\BZ$. In fact, we claim that there are dependents of all of $F_i$ in the $(i\ell+i-1)$-th column, and that there are dependents of all of $S_{i,j}$ in the $(\alpha_{i,j}\ell+\alpha_{i,j}-1)$-th column. 
     
     But this is evident from the calculations we did earlier. Indeed, for $1\le i\le t$,
     \[[x^{i\ell+i-1}](C+V)(x)=q^{\ol F_i+\ol W_i},\]
     and by construction $\ol F_i$ is minimal with respect to $\ol F_i+\ol W_i=\BZ$, which implies that every element of $\ol F_i$ has a dependent in the $(i\ell+i-1)$-th column. Similarly, for $t+1\le \alpha(i,j)\le N$, 
     \[[x^{\alpha_{i,j}\ell+\alpha_{i,j}-1}](C+V)(x)=q^{\ol K_i\sqcup \bigsqcup_\mu (\ol F_\mu+\ol W_{i,j;\mu})\sqcup\bigsqcup_\nu (\ol K_\nu+\ol U_{i,j;\nu})},\]
     where
     \[\bigcup_{\mu=1}^t (\ol F_\mu+\ol W_{i,j;\mu})\cup\bigcup_{\nu=1}^r (\ol K_\nu+\ol U_{i,j;\nu})=\BZ\sm \ol S_{i,j}\]
     and $\ol S_{i,j}\subseteq\ol K_i$, so that $S_{i,j}$ is covered exactly once in the expression $\ol K_i\sqcup \bigsqcup_\mu (\ol F_\mu+\ol W_{i,j;\mu})\sqcup\bigsqcup_\nu (\ol K_\nu+\ol U_{i,j;\nu})$, so that there are dependent elements of $\ol S_{i,j}$ in column $\alpha_{i,j}\ell+\alpha_{i,j}-1$. 
     
     Unioned over all $i,j$, this gives dependent elements for all of $\bigcup_{i,j}S_{i,j}=K$ and $\bigcup_i F_i=F$, so that there are dependent elements for all of $C$, which concludes our check that $C$ is minimal.
\end{proof}

\section{Comments, Questions, and Further Directions}

\subsection{Variations of Main Results}
    In our proof for Theorem 4, note that the place where we crucially used that $B$ is finite is in saying that $m\BN\cup B$ is bounded below, and therefore $S_i$ is bounded below, so if $F+W_i=m\BZ\sm S_i$, then necessarily $F+W_i$ contains a infinite set extending into the negative direction, so that the BL lemma applies. If we drop the assumption that $B$ is finite, this is no longer guaranteed, and we can no longer ensure that $F+W_i=m\BZ\sm S_i$ in a way that gives the dependents of $F$. However we can sidestep this issue by dedicating a separate column to the dependents of $F$: since $\ol F$ is finite, by the theorem of Kwon it arises as a MAC in $\BZ$, and we can instead use $\ol F+\ol W=\BZ$ to give the dependent elements of $F$. Modifying the proof appropriately so that there is now $n+1$ columns of dependents rather than $n$, we can obtain that
    \begin{theorem}
        Let $|B_{/m}|=|F_{/m}|=|A|=1$ with $F_{/m}=\{\wt f\}$, without loss of generality let $A=\{0\}$, and let $B$ be infinite. Then the existence of a set cover $\{S_i\}$,
        \[S_1\cup\cdots\cup S_n=m\BN\cup B,\]
        such that each member $S_i$ has
        \[m\BZ\sm S_i=F+W_i\]
        for some $W_i\subseteq\BZ$, implies that if 
        \[m\ge \wt f+2\wt f\floor*{\frac{n+1}{\wt f}}+\on{mod}_{\wt f}(n+1),\]
        then
        \[C=m\BN\cup B\cup F\]
        arises as a MAC.
    \end{theorem}
    Note that this bound is the same as the one in Theorem 4 except $n$ is replaced with $n+1$, where the $+1$ is for the extra column dedicated to $F$. 
    
    One might then ask what happens if the $|B|<\infty$ condition is weakened in the general case. We do not see an immediate solution, as Theorem 6 was derived from Theorem 7 by taking the worst-case scenario in which we dedicate a column for every member of $B$, but if $B$ is infinite then this approach no longer works. We hence pose the following question:
    \begin{question*}
        Is some variant of Theorem 6 true for infinite $B$? That is, even if $B$ is infinite, is it still true that every eventually periodic set is eventually a MAC?
    \end{question*}

\subsection{Regarding the Formal Power Series}
    We remark that, in the ring $\cal Q$, the only units are $\pm q^0$. 

    Since no two nonempty sets satisfy $A\oplus B=\emptyset$, we have that no two nonzero $q^A,q^B$ satisfy $q^Aq^B=q^\emptyset=0$, so that $\cal Q$ is an integral domain. Then we can consider the fraction field $\Frac \cal Q$, as well as generating series $(\Frac\cal Q)[[x]]$ with coefficients in this fraction field. The classical result from elementary formal power series tells us that
    \begin{observation*}
        A power series $S(x)\in(\Frac\cal Q)[[x]]$ has a uniquely determined multiplicative inverse in this ring if and only if $[x^0]S(x)\neq 0$. 
    \end{observation*}
    For the same reason, $[x^0]S(x)\neq 0$ implies $S(x)\in(\Frac\cal Q)[[x]]/\wangle{x^m-q^1}$ has a multiplicative inverse. In fact, this is true even if $[x^0]S(x)=0$, as long as $S(x)$ is not identically zero. Indeed, in the ring $S(x)\in(\Frac\cal Q)[[x]]/\wangle{x^m-q^1}$, note that $x^i$ has the multiplicative inverse $q^{-1}x^{m-i}$; then, given a $S(x)$ with zero constant term, we can write it as $S(x)=x^k\pr*{\sum_{i=0}^{m-1} q^{A_i}x^i}$, where $q^{A_0}\neq 0$; then $S(x)^{-1}=q^{-1}x^{m-k}\pr*{\sum_{i=0}^{m-1} q^{A_i}x^i}^{-1}$. Hence
    \begin{observation*}
        Any nonzero power series $0\neq S(x)\in(\Frac\cal Q)[[x]]/\wangle{x^m-q^1}$ has a multiplicative inverse in this ring.
    \end{observation*}
    
    One might wonder if this multiplicative inverse is unique like it is for the classical setting of formal series over fields. By expanding the equation
    \[(q^{A_0}x^0+\cdots+q^{A_{m-1}}x^{m-1})(q^{B_0}x^0+\cdots+q^{B_{m-1}}x^{m-1})=1,\]
    one can see that the relevant system of equations is
    \begin{align*}
        \begin{pmatrix}
            q^{A_0} & q^{A_{m-1}+1} & q^{A_{m-2}+1} & q^{A_{m-3}+1} & \cdots & q^{A_1+1}\\
            q^{A_1} & q^{A_0} & q^{A_{m-1}+1} & q^{A_{m-2}+1} & \cdots & q^{A_2+1}\\
            q^{A_2} & q^{A_1} & q^{A_0} & q^{A_{m-1}+1} & \cdots & q^{A_3+1}\\
            \vdots & \vdots &\vdots & \ddots & \ddots & \vdots\\
            q^{A_{m-2}} & q^{A_{m-3}} & q^{A_{m-4}} & \cdots & q^{A_0} & q^{A_{m-1}+1}\\
            q^{A_{m-1}} & q^{A_{m-2}} & q^{A_{m-3}} & \cdots &q^{A_1} & q^{A_0}
        \end{pmatrix}\begin{pmatrix} q^{B_0}\\ q^{B_1}\\q^{B_2}\\ \vdots\\ q^{B_{m-2}} \\ q^{B_{m-1}} \end{pmatrix}=\begin{pmatrix} q^0\\ 0\\0\\ \vdots\\0 \\ 0 \end{pmatrix}.
    \end{align*}
    (The presence of the extra $q^1$'s on the upper triangle comes from the fact that $x^m=q^1$.) Then a multiplicative inverse in $(\Frac\cal Q)[[x]]/\wangle{x^m-q^1}$ is unique if and only if this Toeplitz matrix on the left (which has entries in this strange field $\Frac\cal Q$)
    \[T\coloneqq \begin{pmatrix}
            q^{A_0} & q^{A_{m-1}+1} & q^{A_{m-2}+1} & q^{A_{m-3}+1} & \cdots & q^{A_1+1}\\
            q^{A_1} & q^{A_0} & q^{A_{m-1}+1} & q^{A_{m-2}+1} & \cdots & q^{A_2+1}\\
            q^{A_2} & q^{A_1} & q^{A_0} & q^{A_{m-1}+1} & \cdots & q^{A_3+1}\\
            \vdots & \vdots &\vdots & \ddots & \ddots & \vdots\\
            q^{A_{m-2}} & q^{A_{m-3}} & q^{A_{m-4}} & \cdots & q^{A_0} & q^{A_{m-1}+1}\\
            q^{A_{m-1}} & q^{A_{m-2}} & q^{A_{m-3}} & \cdots &q^{A_1} & q^{A_0}
        \end{pmatrix}\]
    is invertible. The author wonders if some type of complexification is possible, so that the Gershgorin circle theorem (or some appropriate variant) becomes applicable. 
    
    We suppose it is possible that somehow this matrix is always invertible. For example, in the case $m=2$, the determinant of this matrix is
    \[\det T=q^{A_0\oplus A_0}-q^{A_1\oplus A_1\oplus 1},\]
    which we see cannot be zero since $A_0\oplus A_0=A_1\oplus A_1\oplus 1$ is impossible, as the smallest number in $A_0\oplus A_0$ (which must be the sum of the smallest number in $A_0$ with itself) must be even, while the smallest number in $A_1\oplus A_1\oplus 1$ must be odd. Hence 
    \begin{observation*}
    In the case $m=2$, all multiplicative inverses in $(\Frac\cal Q)[[x]]/\wangle{x^m-q^1}$ are unique. 
    \end{observation*}
    But even in the case $m=3$ this becomes more unwieldy. Indeed, for $m=3$ the determinant becomes
    \[\det T=q^{A_0}q^{A_0}q^{A_0}+q^{A_1}q^{A_1}q^{A_1+1}+q^{A_2}q^{A_2+1}q^{A_2+1}-3q^{A_0}q^{A_1}q^{A_2}q^1,\]
    so that 
    \[\det T=0\iff (A_0\oplus A_0\oplus A_0)\sqcup(A_1\oplus A_1\oplus A_1\oplus 1)\sqcup (A_2\oplus A_2\oplus A_2\oplus 2)=3(A_0\oplus A_1\oplus A_2\oplus 1).\]
    We can treat this case in the same way as before: letting the minimal element of $A_i$ be $a_i$ with multiplicity\footnote{Not conflicting with the $m$ which denotes which $\Xi_m$ we are in.} $m_i$, we see that the minimum element of the left-hand side is
    \[\min \bigpr{A_0^{\oplus 3}\sqcup (A_1^{\oplus 3}\oplus1)\sqcup (A_2^{\oplus 3}\oplus 1)}=\min\{3a_0,3a_1+1,3a_2+2\},\ \te{with multiplicity } m_i^3;\]
    on the other hand, the minimal element of the right hand side is
    \[\min \bigpr{3(A_0\oplus A_1\oplus A_2\oplus 1)}=a_0+a_1+a_2+1,\ \te{with multiplicity }m_0m_1m_2.\]
    For $\det T=0$ to be true, the right-hand-side minimum must agree with the left-hand-side minimum, which is one of the stated three things. If $a_0+a_1+a_2+1=3a_0$ is the common minimum of both sides, then $a_1+a_2+1=2a_0$, i.e. $1=(a_0-a_1)+(a_0-a_2)$, so that $a_0$ must be greater than one of $a_1$ and $a_2$, contradicting the minimality of $3a_0$. If $a_0+a_1+a_2+1=3a_1+1$ is the common minimum of both sides, then $a_0+a_2=2a_1$. i.e. $0=(a_1-a_0)+(a_1-a_2)$, which also contradicts\footnote{If $a_1\ge a_0$, then $3a_0<3a_1+1$, contradiction. Otherwise, we must have $a_1>a_2$, so that $3a_1+1>3a_2+2$, also contradiction.} the minimality of $3a_1+1$. Lastly, if $a_0+a_1+a_2+1=3a_2+2$ is the common minimum, then $a_0+a_1=2a_2+1$, i.e. $1=(a_0-a_2)+(a_1-a_2)$, also contradicting\footnote{$a_0\le a_2$ implies $3a_0<3a_2+2$, contradicting minimality, so we must have $a_0>a_2$. But $a_1\le a_2$ also implies $3a_1+1<3a_2+2$, also contradiction; hence also $a_1>a_2$. But the sum of two positive numbers cannot be 1.} the minimality of $3a_2+2$. Hence the two minima cannot possibly be equal, and so we conclude $\det T\neq 0$:
    \begin{proposition}
        For $m\le 3$, all multiplicative inverses for nonzero elements in $(\Frac\cal Q)[[x]]/\wangle{x^m-q^1}$ exist and are unique.  
    \end{proposition}
    As a remark, it is not enough to check the sizes (with multiplicity taken into account of course) of both sides, since the sizes are
    \[|A_0|^3+|A_1|^3+|A_2|^3\quad\te{and}\quad 3|A_0||A_1||A_2|,\]
    which is precisely AM-GM and are equal precisely when $|A_0|=|A_1|=|A_2|$. 
    
    However, one could try to use this idea of comparing sizes to derive sufficient conditions for unique inverses. Taking the sizes of the entries in the Toeplitz matrix, we obtain that if the determinant of the symmetric\footnote{This is because $|A\oplus 1|=|A|$.} Toeplitz matrix
    \[\det\begin{pmatrix}  
        |A_0| & |A_{m-1}| &\cdots& \cdots & |A_1|\\
        |A_1| & |A_0| & |A_{m-1}| & \cdots &|A_2|\\
        \vdots & |A_1| & \ddots & \ddots& \vdots \\
        \vdots  & \vdots &\ddots &|A_0| & |A_{m-1}|\\
        |A_{m-1}|&|A_{m-2}|&\cdots&|A_1|&|A_0|
    \end{pmatrix}\]
    is nonzero, then $A(x)$ has a unique multiplicative inverse in $(\Frac\cal Q)[[x]]/\wangle{x^m-q^1}$. This matrix has integer entries, so we may now apply familiar results; for example, applying the Gershgorin circle theorem\footnote{As a reminder, this theorem states that, for a complex matrix $A=\{a_{ij}\}$, every eigenvalue of $A$ lies in at least one of the discs $B_{\sum_{j\neq i}|a_{ij}|}(a_{ii})\subset\BC$.} [G], one obtains that
    \begin{proposition}
        If 
        \[|A_0|>\sum_{i=1}^{m-1}|A_i|,\]
        then $A(x)$ has a unique multiplicative inverse in $(\Frac\cal Q)[[x]]/\wangle{x^m-q^1}$.
    \end{proposition}
    
    The case $m=3$ might have given us hope that there is some pattern to be had, but $\det T$ for $m=4$ is
    \[\det T=q^{A_0^4}-4q^{A_0^2}q^{A_1}q^{A_3}q^1-2q^{A_0^2}q^{A_2^2}q^1+4q^{A_0}q^{A_1^2}q^{A_2}q^1\]
    \[\qquad\qquad\qquad\qquad\qquad\qquad\qquad+4q^{A_0}q^{A_2}q^{A_3^2}q^2-q^{A_1^4}q^1+2q^{A_1^2}q^{A_3^2}q^2-4q^{A_1}q^{A_2^2}q^{A_3}q^2+q^{A_2^4}q^2-q^{A_3^4}q^3,\]
    which is far messier. We leave open the following question:
    \begin{question*}
        Under what general circumstances (perhaps always) are multiplicative inverses unique in $(\Frac\cal Q)[[x]]/\wangle{x^m-q^1}$?
    \end{question*}
    
    As a fun example, we can consider the equation
    \[1+x+x^2+\cdots=\frac{1}{1-x}.\]
    When taken to $\Xi_m$, the left hand side is
    \begin{align*}
        1+x+x^2+\cdots&=x^0+x^1+q^1x^0+q^1x^2+q^2x^0+q^2x^1+\cdots\\
        &=(q^0+q^1+q^2+\cdots)x^0+(q^0+q^1+q^2+\cdots)x^1\\
        &=q^\BN x^0+q^\BN x^1\\
        &=(\BN)(x).
    \end{align*}
    On the other hand,
    \begin{align*}
        1-x&=q^0x^0-q^0x^1\\
        &=(\{0\})(x)-(\{1\})(x)\\
        &=(\{0\}\sqminus\{1\})(x),
    \end{align*}
    so that the equation $(1+x+x^2+\cdots)(1-x)=1$ becomes
    \begin{align*}
        1&=(\BN)(x)(\{0\}\sqminus\{1\})(x)\\
        &=(\BN\oplus (\{0\}\sqminus\{1\}))(x)\\
        &=((\BN\oplus\{0\})\sqminus(\BN\oplus\{1\}))(x)\\
        &=(\{1\})(x),
    \end{align*}
    so that 
    \[(1+x+x^2+\cdots)(1-x)=1\llra (\BN\oplus\{0\})\sm (\BN\oplus\{1\})=\{1\}.\]
    
    Since we've seen that multiplication of these formal power series corresponds to the (disjoint) Minkowski sum and addition corresponds to the disjoint union, one may ask what the functional composition corresponds to. The answer is not clear to us, and we leave this open:
    \begin{question*}
        What does functional composition of these formal power series correspond to set-theoretically?
    \end{question*}
    
\subsection{Vaguer Questions and Directions}
    In the construction of these formal power series, our idea was that the information of a set in $Z_m\cong\BZ$ is the same as $\cal P(\BZ)^m$, that is the information of $m$ elements in the power set of $\BZ$. The information of $\cal P(\BZ)$ is then encoded in the coefficients of our formal power series, and the exponents of $x$ indicate which position this set is at in our $m$ elements of $\BZ$. One might ask how to generalize this idea: in general, one might roughly have different exponential symbols $\exp_i(S_i)\in \cal Q_i$ which do not combine, and coefficients in a power series might look like products $\prod_i\exp_i(S_i)$ of these symbols, and each term might have products of powers of variables attached of form $\prod_i x_i^{n_i}$. We suspect these formal power series will only make sense for finitely-generated abelian groups.
    \begin{direction*}
        To further investigate and make precise these generalized formal power series as well as their combinatorial set-theoretic interpretations.
    \end{direction*}
    
    In a similar vein, looking at our setup for the statement of Theorem 7, our idea was to take a collection of points in $\BZ^2$ satisfying some conditions, quotient out by some sequence of ``increasing'' submodules (in our case ``increasing'' meant $\BZ\wangle{(m,-1)}$ where $m$ is increasing), and ask at what point does the image become a MAC. But why this particular sequence of submodules? We suspect that, for some appropriately defined family of ``increasing'' submodules, Theorem 7 is still true.
    \begin{question*}
        How should one define these families of ``increasing'' submodules so that some appropriate variant of Theorem 7 is still true? And how ought we to modify the formal power series in this framework?
    \end{question*}
    
    One can take these questions even further. But why should 2 be special? 
    \begin{question*}
    What if instead we considered a collection of points in some other (finitely-generated) abelian group, quotiented out by some sequence of ``increasing'' submodules (we guess of the same rank), and asked at what point does the image become a MAC? And how ought we to define the formal power series in this framework?
    \end{question*}
    
    Lastly, one might consider similar questions for nonabelian groups, but without the structure afforded by finitely-generated abelian groups this question is rather unclear. We would also like to remark that, in the case of formal power series for enumerative purposes, there was a categorification incarnated in the form of the theory of combinatorial species wherein lied a vast and rich theory. Although much more of a ``long shot'', one might ask whether the formal power series defined in the present paper also admitted such a categorification. In the species case, this was achieved roughly by considering the automorphism groups of the exponents (where $x^n$ represented $[n]$) and gluing them together into a category; in our case, it is unclear what, if any, automorphism groups should be considered.
    
\section{Appendix: An Explanation of the Operations of Series and Proof of Proposition 1 in a Special Case}
We explain Definition 9.
\begin{proof}[Explanation of Definition]
    We have secretly hidden some content in claiming $A(x)+B(x)=(A\sqcup B)(x)$ and $A(x)B(x)=(A+ B)(x)$, but the amount of content is epsilon and is evident from FOIL (distributive property of multiplication). Indeed, writing $A=A_1\sqcup\cdots\sqcup A_r$ and $B=B_1\sqcup\cdots\sqcup B_t$, it is clear that $A\sqcup B=A_1\sqcup\cdots\sqcup A_r\sqcup B_1\sqcup\cdots\sqcup B_t$, i.e.
    \[(A\sqcup B)(x)=\sum_{i=1}^rq^{A_i}x^{a_i}+\sum_{j=1}^t q^{B_j}x^{b_j}=A(x)+B(x).\]
    
    Now that we've shown the first equality, to show that $A(x)B(x)=(A+B)(x)$, let us first show the basic case when $A=A_1$ and $B=B_1$ are each concentrated in a single column. In that case $A(x)=q^Ax^a$ and $B(x)=q^Bx^b$, whereupon it is obvious that $A+B\llra (\ol A+\ol B,\wt a+\wt b)$, so that $(A+B)(x)=q^{\ol A+\ol B}x^{\wt a+\wt b}=q^{A+B}x^{a+b}$, while $A(x)B(x)=q^Ax^aq^Bx^b$, so that
    \[(A+B)(x)=A(x)B(x).\]
    Now that the base case is established, to get the general case we can see that
    \begin{align*}
        A+B&=   (A_1\sqcup\cdots\sqcup A_r)+B\\
     &=(A_1+B)\sqcup\cdots\sqcup(A_r+B)\\
        (A+B)(x)&=\bigpr{(A_1+B)\sqcup\cdots\sqcup(A_r+B)}(x)\\
        &=(A_1+B)(x)+\cdots+(A_r+B)(x)\\
        &=\bigpr{A_1+(B_1\sqcup\cdots\sqcup B_t)}(x)+\cdots +\bigpr{A_r+(B_1\sqcup\cdots\sqcup B_t)}(x)\\
        &=\bigpr{(A_1+B_1)\sqcup\cdots\sqcup (A_1+B_t)}(x)+\cdots +\bigpr{(A_r+B_1)\sqcup\cdots\sqcup (A_r+B_t)}(x)\\
        &=(A_1+B_1)(x)+\cdots+(A_1+B_t)(x)+\cdots (A_r+B_1)(x)+\cdots+(A_r+B_t)(x)\\
        &=A_1(x)B_1(x)+\cdots+A_1(x)B_t(x)+\cdots+A_r(x)B_1(x)+\cdots A_r(x)B_t(x)\\
        &=\bigpr{A_1(x)+\cdots+A_r(x)}\bigpr{B_1(x)+\cdots+B_t(x)}\\
        &=A(x)B(x)
    \end{align*}
    \end{proof}
    
    We now prove Proposition 1 in the case that $B$ is infinite.
    \begin{proof}[Proof of Proposition 1]
        If $B$ is infinite and $m\BN\cup B\neq m\BZ$, we can still proceed as follows. Being a singleton, $\{f\}$ has another set $W$ such that $\{f\}+W=m\BZ_-\sm B$ (this set is not empty since $m\BN\cup B\neq m\BZ$). Then 
    \[C+\pr*{\{0\}\cup W\cup\bigcup_{i=1}^{m-1}(m\BZ+i)\sm(m\BZ+m-\wt f)}=\BZ\]
    realizes $C$ as a MAC. Indeed, in this expression, $C+\{0\}$ ensures that $m\BN\cup B$ has dependent elements, and $C+W$ ensures that $\{f\}$ has dependent elements. It is easy to see that this covers all of $\BZ$, as well as that the 0-th column contains dependent elements for all of $C$. 
    
    Even if $m\BN\cup B=m\BZ$, as long as $m\ge 3$, we can still consider 
    \[C+\pr*{\bigcup_{i=0}^{m-2}(i+im)\cup (m\BZ+m-1-\wt f)        }=\BZ.\]
    The dependent elements of $m\BN\cup B$ are given by the translates by $i+im$, and the dependent elements of $\{f\}$ are given by the translates by $m\BZ+m-1-\wt f$. 
    
    However, if both $m\BN\cup B=m\BZ$ and $m=2$, it is easy to see by inspection that this set cannot be a MAC; we cannot cover all of $\BZ$ while maintaining a dependent element of $\{f\}$, for instance.
    % \[C+\Bigpr{\bigpr{[0,m-2]\sm (m-1-\wt f)}\cup(m\BZ+m-1-\wt f)}=\BZ\]
    \end{proof}
     
\section{Acknowledgements}
    We would like to thank the Duluth REU Program, where this present research was conducted under grant numbers NSF-DMS 1949884 and NSA H98230-20-1-0009, for a wonderful nurturing experience and environment and Professor Joe Gallian for putting everything together and making all of this possible. We would also like to thank Professor Gallian and the REU participants and advisors as well as all the visitors for their helpful comments during the talks; in particular we are incredibly thankful and indebted to Amanda Burcroff, whose paper this was based off of, for introducing us to the subject and for her many insightful and helpful comments and suggestions. Lastly we would like to thank our grandparents for their unwavering support.\begin{CJK*}{UTF8}{gbsn} 愚孙无翁慈\ 无以至今日。\end{CJK*}
     
\section{References}

\begin{flushleft}
\begin{tabular}{ll}
\text{[AKL]} & Noga Alon, Noah Kravitz, and Matt Larson. Inverse problems for minimal complements\\
&and maximal supplements. Preprint at \url{https://arxiv.org/abs/2006.00534}, (2020).\\
\text{[BL]} & Amanda Burcroff and Noah Luntzlara. Sets arising as minimal additive complements in\\
& the integers. Preprint at \url{https://arxiv.org/abs/2006.12481}, (2020).\\
\text{[BSa]} & Arindam Biswas and Jyoti Prakash Saha. On non-minimal complements. Preprint avail-\\
&able at \url{https://arxiv.org/abs/2007.08507}, (2020).\\
\text{[BSb]} & Arindam Biswas and Jyoti  Prakash Saha. Asymptotic behaviour of minimal compleme-\\
&nts. Preprint available at \url{https://arxiv.org/abs/2007.14389}, (2020).\\
\text{[CY]} & Yong-Gao Chen and Quan-Hui Yang. On a problem of Nathanson related to minimal ad-\\
&ditive complements. \textit{SIAM Journal on Discrete Mathematics}, 26(4): 1532--1536, (2012).\\
\text{[G]} & \"Uber die Abgrenzung der Eigenwerte einer Matrix. Izv. Akad. Nauk. USSR Otd. Fiz.\\
& -Mat. Nauk 6 (1931), 749--754.\\
\text{[K]} & Andrew Kwon. 
A note on minimal additive complements of integers. \emph{Discrete Mathem-}\\
& \emph{atics},
342(7): 1912--1918, (2019). Preprint at \url{https://arxiv.org/abs/1708.01287}.\\
\text{[KSY]} & S\'andor Kiss, Csaba S\'andor, and Quan-Hui Yang. On minimal additive complements of\\
&integers. \textit{Journal of Combinatorial Theory, Series A}, 162: 344--353, (2019).\\
\text{[N]} & Melvyn Nathanson. Problems in additive number theory, IV: Nets in groups and shorte-\\
& st length $g$-adic representations. \textit{International Journal of Number Theory}, 7(8): \\
& 1999--2017, (2011).
\end{tabular}
\end{flushleft}

\end{document}